\documentclass[12pt]{article}%
   
\usepackage{amsmath,enumerate}
\usepackage{amsfonts}
\usepackage{amssymb}

\newtheorem{theorem}{Theorem}[section]

\newtheorem{conjecture}[theorem]{Conjecture}
\newtheorem{corollary}[theorem]{Corollary}

\newtheorem{lemma}[theorem]{Lemma}

\newtheorem{problem}[theorem]{Problem}

\newenvironment{proof}[1][Proof]{\noindent\textbf{#1.} }
{\hfill \ \rule{0.5em}{0.5em}}

\newcommand{\integers}{\mathbb{Z}}

\newcommand{\ind}{1\hspace{-2.6mm}{1}} 

\begin{document}

\title{Upper and lower bounds on $B_k^+$-sets}

\author{Craig M.~Timmons\thanks{Department of Mathematics, University of California San Diego, La Jolla, CA 92093, ctimmons@ucsd.edu.}}

\maketitle

\begin{abstract}
Let $G$ be an abelian group.  
A set $A \subset G$ is a \emph{$B_k^+$-set} if whenever $a_1 + \dots + a_k = b_1 + \dots + b_k$ with $a_i , b_j \in A$ 
there is an $i$ and a $j$ such that $a_i = b_j$.  If $A$ is a $B_k$-set then it is also a $B_k^+$-set but the converse is not 
true in general.  Determining the largest size of a $B_k$-set 
in the interval $\{1, 2, \dots , N \} \subset \integers$ or in the cyclic group $\integers_N$ is a well studied problem.  In this paper 
we investigate the corresponding problem for $B_k^+$-sets.  We prove non-trivial upper bounds on the maximum size of a 
$B_k^+$-set contained in the interval $\{1, 2, \dots , N \}$.  
For odd $k \geq 3$, we construct $B_k^+$-sets that have more elements than the $B_k$-sets constructed by Bose and Chowla.  
We prove a $B_3^+$-set $A \subset \integers_N$ has at most $(1 + o(1))(8N)^{1/3}$ elements.  Finally we obtain new upper 
bounds on the maximum size of a $B_k^*$-set $A \subset \{1,2, \dots , N \}$, a problem first investigated by Ruzsa.          
\end{abstract}

%%%%%%%%%%%%%%%%%%%%%%%%%%%%%%

\section{Introduction}

%%%%%%%%%%%%%%%%%%%%%%%%%%%%%%%%

Let $G$ be an abelian group.  A set $A \subset G$ is a \emph{$B_{k}^{+}$-set} if   
\begin{equation}\label{intro eq 1}
a_1 + \dots + a_k = b_1 + \dots + b_k ~~ \textrm{with} ~~ a_1, \dots , a_k,b_1, \dots , b_k \in A
\end{equation}
implies $a_i = b_j$ for some $i$ and $j$.  $A$ is a \emph{$B_k$-set} if (\ref{intro eq 1}) implies $(a_1, \dots , a_k)$ is a permutation of $(b_1, \dots , b_k)$.  If $A$ is a $B_k$-set then $A$ is also a $B_{k}^{+}$-set but in general the converse is not true.  $B_2$-sets are often called \emph{Sidon sets} and have received much attention since they were first studied by Erd\H{o}s and Tur\'{a}n \cite{ET} in 1941.    
Let $F_k (N)$ be the maximum size of a $B_k$-set $A \subset [N]$ and let 
$C_k (N)$ be the maximum size of a $B_k$-set $A \subset \integers_N$.  If $A \subset \integers_N$ is a $B_k$-set then 
$A$ is also a $B_k$-set when viewed as a subset of $\integers$ so for any $k \geq 2$, $C_k (N) \leq F_k (N)$.

Erd\H{o}s and Tur\'{a}n proved $F_2(N) \leq N^{1/2} + O(N^{1/4})$.  Their argument was used by Lindstr\"{o}m \cite{L} to show $F_2(N) \leq N^{1/2} + N^{1/4} +1$.  In 2010, Cilleruelo \cite{C1} obtained $F_2(N) \leq N^{1/2} + N^{1/4} + \frac{1}{2}$ as a consequence of a more general result and this is the best known upper bound on $F_2(N)$.  By counting differences 
$a - b$ with $a \neq b$, it is easy to prove $C_2(N) \leq \sqrt{N} +1$.  There are several constructions of dense $B_2$-sets (see \cite{S}, \cite{BC}, \cite{R}) that show $C_2(N) \geq N^{1/2}$ for infinitely many $N$.  
It follows that $F_2 (N) \sim \sqrt{N}$ and $\limsup \frac{C_2(N)}{ \sqrt{N} } = 1$.    

For $k \geq 3$, bounds on $F_k(N)$ and $C_k(N)$ are not as precise.  For each $k \geq 2$ and prime power $q$, Bose and Chowla \cite{BC} constructed a $B_k$-set $A \subset \integers_{q^k-1}$ with $|A| = q$ so that  
\begin{equation*}
(1 + o(1)) N^{1/k}  \leq F_k (N).
\end{equation*}
The current upper bounds on $F_k(N)$ and $C_k(N)$ do not match this lower bound for any $k \geq 3$.  If $A \subset [N]$ is a $B_k$-set then each $k$-multiset in $A$ gives rise to a unique sum in $\{1, \dots , kN \}$ so ${ |A| + k - 1 \choose k } \leq kN$ hence $F_k(N) \leq ( k! \cdot k N )^{1/k}$.  Similar counting 
shows $C_k (N) \leq ( k! N)^{1/k}$.  By considering differences one can improve these bounds.  We illustrate this idea with an example that is relevant to our results.  Let $A \subset \integers_N$ be a $B_3$-set.  There are ${ |A| \choose 2}( |A| - 2)$ 
sums of the form $a_1 + a_2 - a_3$ where $a_1, a_2$, and $a_3$ are distinct elements of $A$.  
It is not hard to check that each $n \in \integers_N$ has at most one representation as $n = a_1 + a_2 - a_3$ with 
$\{a_1, a_2 \} \in A^{(2)}$ and $a_3 \in A \backslash \{a_1, a_2 \}$.  
 This implies ${ |A| \choose 2}( |A| -2) \leq N$ so $|A| \leq (2N)^{1/3}+2$.  In general, for any $k \geq 2$ 
\begin{equation}\label{trivial ck}
C_k(N) \leq  \left( \left\lfloor \frac{k}{2} \right\rfloor ! \left\lceil \frac{k}{2} \right\rceil ! N \right)^{1/k} + O_k (1)
\end{equation}  
and 
\begin{equation}\label{trivial fk}
F_k (N) \leq  \left(  \left\lfloor \frac{k}{2} \right\rfloor ! \left\lceil \frac{k}{2} \right\rceil ! \cdot k N \right)^{1/k}+O_k(1).
\end{equation}
These bounds were first obtained by Jia \cite{Jia} in the even case and Chen \cite{Chen} in the odd case.        
The best upper bounds on $F_k(N)$ are to due to Green \cite{G}.  For every $k \geq 2$, (\ref{trivial fk}) has been improved (see for example \cite{G} or \cite{C}) but there is no value of $k \geq 3$ for which (\ref{trivial ck}) has been improved.    
This is interesting since all of the constructions take place in cyclic groups and provide lower bounds on $C_k(N)$.    
For other bounds on $B_k$-sets the interested reader is referred to Green \cite{G}, Cilleruelo \cite{C}, O'Bryant's survey \cite{OB}, or the book of Halberstam and Roth \cite{HR}.   

Now we discuss $B_k^+$-sets.  Write $F_k^+ (N)$ for the maximum size of a $B_k^+$-set $A \subset [N]$ and 
$C_k^+(N)$ for the maximum size of a $B_k^+$-set $A \subset \integers_N$.  Ruzsa \cite{R} proved that a set $A \subset [N]$
with no solution to the equation $x_1 + \dots + x_k = y_1 + \dots + y_k$ in $2k$-distinct integers has at most 
$(1 + o(1))k^{2-1/k}N^{1/k}$ elements.  Call such a set a $B_k^*$-set and define $F_k^* (N)$ in the obvious way.  Any 
$B_k^+$-set is also a $B_k^*$-set so $F_k^+ (N) \leq F_k^* (N)$.  Using the constructions of Bose and Chowla \cite{BC} and Ruzsa's Theorem 5.1 of \cite{R}, we get for every $k \geq 3$
\begin{equation*}
(1 + o(1)) N^{1/k}  \leq F_k (N) \leq F_k^+ (N) \leq F_k^* (N) \leq (1 + o(1)) k^{2 - 1/k }N^{1/k}.
\end{equation*} 

In this paper we improve this upper bound on $F_k^+(N)$ and $F_k^* (N)$, and improve the lower bound 
on $F_k^+ (N)$ for all odd $k \geq 3$.  We also prove a non-trivial upper bound on $C_3^+(N)$.  

Our first result is a construction which shows that for any odd $k \geq 3$, there is a $B_k^+$-set in $[N]$ that has more elements than any known $B_k$-set contained in $[N]$.  

\begin{theorem}\label{lower bound theorem}  
For any prime power $q$ and odd integer $k \geq 3$, there is a $B_k^+$-set $A \subset \integers_{2 (q^k-1) }$ with $|A| = 2q$.  
\end{theorem}

Using known results on densities of primes (see \cite{BHP} for example), Theorem~\ref{lower bound theorem} implies 
\begin{corollary}\label{lb cor}
For any integer $N \geq 1$ and any odd integer $k \geq 3$, 
\begin{equation*}
F_k^+ (N) \geq (1 + o(1)) 2^{1 - 1/k}N^{1/k}.
\end{equation*}
\end{corollary} 

Green proved $F_3(N) \leq (1 + o(1)) (3.5N)^{1/3}$.  We will use a Bose-Chowla $B_3$-set to construct a $B_{3}^{+}$-set $A \subset [2q^3]$ with $|A| = 2q = (4 \cdot 2q^3)^{1/3}$.
Putting the two results together we see that $A$ is denser than any $B_3$-set in $[2q^3]$ for sufficiently large prime powers $q$.  Ruzsa calls a $B_2^*$-set a \emph{weak Sidon set} and proves that a weak Sidon set $A \subset [N]$ has at most $N^{1/2} + 4N^{1/4} + 11$ elements (see \cite{R}) from which we deduce $F_2(N) \sim F_2^* (N)$.  Our construction and Green's upper bound show that $F_3(N)$ and $F_3^*(N)$ are not asymptotically equal.         

The proof of Theorem~\ref{lower bound theorem} is based on a simple lemma, Lemma~\ref{lemma 1 lb}, which implies 
\begin{equation}\label{F_k}
2  C_k(N) \leq  C_k^+ (2N) ~ \mbox{for any odd} ~ k \geq 3.
\end{equation}
This inequality provides us with a method of estimating $C_k(N)$ by proving upper bounds on $C_k^+ (N)$ for odd $k$.  Our next 
theorem provides such an estimate when $k =3$.  

\begin{theorem}\label{ub 3}
(i) If $A \subset \integers_N$ is a $B_3^+$-set then 
\begin{equation*}
|A|  \leq (1 + o(1)) (8 N)^{1/3}.
\end{equation*}

\noindent
(ii) If $A \subset [N]$ is a $B_3^+$-set then 
\begin{equation*}
|A| \leq (1 + o(1))(18N)^{1/3}.
\end{equation*}

\noindent
(iii) If $A \subset [N]$ is a $B_4^+$-set then 
\begin{equation*}
|A| \leq (1 + o(1))(272N)^{1/4}.
\end{equation*}
\end{theorem}
   
Theorem~\ref{ub 3}(i) and (\ref{F_k}) imply
\begin{corollary}\label{C_3}
If $A \subset \integers_N$ is a $B_3$-set then 
\begin{equation*}
|A| \leq (1 + o(1)) ( 2 N)^{1/3}.
\end{equation*}
\end{corollary}
As shown above, there is a simpler argument that implies this bound.  The novelty here is that our results imply (\ref{trivial ck}) for $k =3$.  It is important to mention that the error term we obtain is larger than the error term in the bound $C_3(N) \leq (2N)^{1/3} +2$.  We feel that any improvement in the leading term of Theorem~\ref{ub 3}(i) or (\ref{trivial ck}) would be significant.        

For $k \geq 5$ we were able to improve the upper bound $F_k^+ (N) \leq (1  +o(1))k^{2 - 1/k}N^{1/k}$ by modifying arguments of Ruzsa.  Our method applies to $B_k^*$-sets and as a consequence we improve the upper bound on $F_k^* (N)$ for all $k \geq 3$.  We state our result only for $k=3$ and for large $k$.  For other small values of $k$ the reader is referred to Table 1 in Section 6.  

\begin{theorem}\label{upper bounds}
If $A \subset [N]$ is a $B_3^*$-set then 
\begin{equation*}
|A| \leq (1 + o(1))( 162N)^{1/3}.
\end{equation*}
If $A \subset [N]$ is a $B_k^*$-set then 
\begin{equation*}
|A| \leq \left( \frac{1}{4} + \epsilon(k) \right)k^2 N^{1/k}
\end{equation*}
where $\epsilon (k) \rightarrow 0$ as $k \rightarrow \infty$.  
\end{theorem}

Our results do not rule out the possibility of $F_k^+(N)$ being asymptotic to $F_k^*(N)$.

\begin{problem}\label{problem 1}
Determine whether or not $F_k^+(N)$ is asymptotic to $F_k^*(N)$ for $k \geq 3$.  
\end{problem}

If $A \subset [N]$ is a $B_k^*$-set then the number of solutions to $2x_1 + x_2 + \dots + x_{k-1} = y_1 + \dots + y_k$ 
with $x_i,y_j \in A$ is $o( |A|^k)$ (see \cite{R}).  A $B_k^*$-set allows solutions to this equation 
with $x_1, \dots , x_{k-1}, y_1 , \dots , y_k$ all distinct but such a solution cannot occur in a $B_k^+$-set.  If it were true that 
$F_k^+(N)$ is asymptotic to $F_k^*(N)$ then this would confirm the belief that it is the sums of $k$ distinct elements of $A$ that control the size of $A$ and the lower order sums should not matter.  Jia \cite{Jia} defines a \emph{semi-$B_k$-set} to be a set $A$
with the property that all sums of $k$ distinct elements from $A$ are distinct.  He states that Erd\H{o}s conjectured \cite{E3} that a 
semi-$B_k$-set $A \subset [N]$ should satisfy $|A| \leq (1 + o(1))N^{1/k}$.  A positive answer to Problem~\ref{problem 1} would be evidence in favor of this conjecture.     

At this time we do not know how to construct $B_{2k}^{+}$-sets or $B_{2k}^{*}$-sets for any $k \geq 2$ that are bigger than the corresponding Bose-Chowla $B_{2k}$-sets.  We were able to construct interesting  $B_4^+$-sets in the non-abelian setting.

Let $G$ be a non-abelian group.  A set $A \subset G$ is a \emph{non-abelian $B_k$-set} if 
\begin{equation}\label{intro eq 2}
a_1 a_2 \cdots  a_k = b_1  b_2  \cdots  b_k ~~ \textrm{with} ~~ a_i,b_j \in A
\end{equation}                     
implies $a_1 = b_1, a_2 = b_2, \dots , a_k = b_k$.  If $A \subset G$ is a non-abelian $B_k$-set then every $k$-letter word in $|A|$ is different so $|A|^k \leq |G|$.  Odlyzko and Smith \cite{OS} proved that there exists infinitely many groups $G$ such that 
$G$ has a non-abelian $B_4$-set $A \subset G$ with $|A| = (1+o(1)) \left( \frac{|G|}{1024}  \right)^{1/4}$.  They actually prove something more general that gives constructions of non-abelian $B_k$-sets for all $k \geq 2$ but this is the only result that we need.  We define a 
\emph{non-abelian $B_k^+$-set} to be a set $A \subset G$ such that (\ref{intro eq 2}) implies $a_i = b_i$ for some 
$ i \in \{1,2, \dots , k \}$.  As in the abelian setting, a non-abelian $B_k$-set is also a non-abelian $B_k^+$-set but the converse is not true in general.  Using a construction 
of \cite{OS}, we prove

\begin{theorem}\label{lb 4}
For any prime $p$ with $p-1$ divisible by 4, there is a non-abelian group $G$ of order $48(p^4 - 1)$ that contains a 
non-abelian $B_4^+$-set $A \subset G$ with 
\begin{equation*}
|A| = \frac{1}{2} (p-1).
\end{equation*}
\end{theorem}

Our result shows that there are infinitely many groups $G$ such that $G$ has a non-abelian $B_4^+$-set $A$ with $|A| = \left( \frac{|G|}{768}  \right)^{1/4} + o( |G|^{1/4} )$.  We conclude our introduction with the following conjecture concerning $B_{2k}^+$-sets.

\begin{conjecture}\label{conjecture}
If $k \geq 4$ is any even integer then there exists a positive constant $c_k$ such that for infinitely many $N$,  
\begin{equation*}
F_{k}^{+}(N) \geq (1 + c_k +o(1) ) N^{1/k}.
\end{equation*}
\end{conjecture} 

If Conjecture~\ref{conjecture} is true with $c_k = 2^{1 - 1/k} - 1$ as in the odd case, then using Green's upper bound $F_4(N) \leq (1 + o(1))(7N)^{1/4}$ we can conclude that $F_4 (N)$ and $F_4^* (N)$ are not asymptotically the same just as in the case when $k =3$.  Our hope is that a positive answer to Conjecture~\ref{conjecture} will either provide an analogue of (\ref{F_k}) for even $k \geq 4$ or a construction of a $B_k^+$-set that does not use Bose-Chowla $B_k$-sets.  

%%%%%%%%%%%%%%%%%%%%%%%%%%%%%%%%%%

\section{Proof of Theorem~\ref{lower bound theorem}}

%%%%%%%%%%%%%%%%%%%%%%%%%%%%%%%%%%%

In this section we show how to construct $B_{k}^{+}$-sets for odd $k \geq 3$.  Our idea is to take a dense $B_k$-set $A$ and a translate of $A$.  

\begin{lemma}\label{lemma 1 lb}
If $A \subset \integers_N$ is a $B_k$-set where $k \geq 3$ is odd, then 
\begin{equation*}
A^{+} := \left\{ a + bN : a \in A, b \in \{0,1 \} \right\}
\end{equation*}
is a $B_k^+$-set in $\integers_{2N}$.
\end{lemma}
\begin{proof}
Let $k \geq 3$ be odd and suppose 
\begin{equation}\label{first eq}
\sum_{i=1}^{k} a_i + b_i N \equiv \sum_{i=1}^{k} c_i + d_i N ~ (\textrm{mod}2N)
\end{equation}
where $a_i, c_i \in A$ and $b_i,d_i \in \{0,1 \}$.  Taking (\ref{first eq}) modulo $N$ gives
\begin{equation*} 
\sum_{i=1}^{k} a_i \equiv \sum_{i=1}^{k} c_ i~ (\textrm{mod}N).
\end{equation*}  
Since $A$ is a $B_k$-set in $\integers_N$, $(a_1, \dots , a_k )$ must be a permutation of $(c_1, \dots , c_k )$.  If we label the $a_i$'s and $c_i$'s so that $a_1 \leq a_2 \leq \dots \leq a_k$ and $c_1 \leq c_2 \leq \dots \leq c_k$ then $a_i = c_i$ for $1 \leq i \leq k$.  Rewrite (\ref{first eq}) as 
\begin{equation*}
\sum_{i=1}^{k} b_i N \equiv \sum_{i=1}^{k} d_i N ~ (\textrm{mod} 2N).
\end{equation*}
The sums $\sum_{i=1}^{k} b_i$ and $\sum_{i=1}^{k} d_i$ have the same parity.  Since $k$ is odd and $b_i, d_i \in \{0,1 \}$, there must be a $j$ such that $b_j = d_j$ so $a_j + b_j N \equiv c_j + d_j N ~ (\textrm{mod} 2N)$.    
\end{proof}

\vspace{1em}

Let $q$ be a prime power, $k \geq 3$ be an odd integer, and $A_k$ be a Bose-Chowla $B_k$-set with $A_k \subset \integers_{q^k - 1}$ (see \cite{BC} for a description of $A_k$).  Let 
\begin{equation*}
A_k^+ = \{ a + b(q^k - 1) : a \in A_k, b \in \{0,1 \} \}.
\end{equation*}
By Lemma~\ref{lemma 1 lb}, $A_k^+$ is a $B_k^+$-set in $\integers_{ 2(q^k-1)}$ and $|A_k^+ | = 2 |A_k| = 2q$ which proves Theorem~\ref{lower bound theorem}.

%%%%%%%%%%%%%%%%%%%%%%%%%%%%%%%%%%

\section{Proof of Theorem~\ref{ub 3}(i)}

%%%%%%%%%%%%%%%%%%%%%%%%%%%%%%%%%%

Let $A \subset \integers_N$ be a $B_3^+$-set.  If $N$ is odd then $2x \equiv 2y ~(\textrm{mod}N)$ implies $x \equiv y ~(\textrm{mod}N)$.  If $N$ is even then $2x \equiv 2y ~(\textrm{mod}N)$ implies 
$x \equiv y~(\textrm{mod}N)$ or $x \equiv y + N/2~(\textrm{mod}N)$.  Because of this, the odd case is quite a bit easier to deal with and so we present the more difficult case.  \underline{In this section $N$ is assumed to be even}.  If $N$ is odd then 
the proof of Theorem~\ref{ub 3}(ii) given in the next section works in $\integers_N$ and the only modification needed
is to divide by $N$ instead of $3N$ when applying Cauchy-Schwarz.  For simplicity of notation, we write $x=y$ rather than $x \equiv y~(\textrm{mod}N)$.  

For $n \in \integers_N$, define 
\begin{equation*}
f(n) = \# \left\{  ( \{ a,c \} , b ) \in A^{(2)} \times A : n = a - b + c, \{ a,c \} \cap \{b \} = \emptyset \right\}.
\end{equation*}
The sum $\sum f(n) (f(n) -1)$ counts the number of ordered pairs $\left(  ( \{a,c \} , b ) , ( \{x,z \} , y ) \right)$ such that the tuples 
$( \{a,c \} , b )$ and $( \{x,z \}, y )$ are distinct and both are counted by $f(n)$.  For each such pair we cannot have $\{ a , c \} = \{ x , z \}$ otherwise the tuples would be equal.  If $( ( \{ a , c \} , b ) , ( \{ x , z \} , y ) )$ is counted by  
$\sum f(n) ( f(n) -1)$ then $a +y + c = x + b + z$.  By the $B_3^+$ property, $\{ a , y , c \} \cap \{ x, b, z \} \neq \emptyset$ so that $\{ a, c \} \cap \{x, z \} \neq \emptyset$ or $b = y$.  The tuples are distinct so both of these cases cannot occur at the same time.  

\vspace{.25em}
\noindent
\textbf{Case 1:}  $\{a,c \} \cap \{ x, z \} \neq \emptyset$ and $b \neq y$.  

Without loss of generality, assume $a = x$.  Cancel $a$ from both sides of the equation $a - b + c = x - y + z$ and solve for $c$
to get $c = b - y +z$.  Here we are using the ordering of the tuples $( ( \{a,c \} , b), ( \{ x, z \} , y ) )$ to designate which element is solved for after the cancellation of the common term.  

If $z = b$ then $c + y = 2b$ and we have a 3-term arithmetic progression (a.p.\ for short).  
The number of trivial 3-term a.p.'s in $A$ is $2|A|$ since for any $a \in A$,
\begin{equation*}
a + a = 2a = 2(a + N/2).
\end{equation*}
Next we count the number of non-trivial 3-terms a.p.'s.    
By non-trivial, we mean that all terms involved in the a.p.\ are distinct and $a + a = 2 (a +N/2)$ is considered trivial.  

If $p + q = 2r$ is a 3-term a.p.\ then call $p$ and $q$ \emph{outer terms}.  
Let $p$ be an outer term of the 3-term a.p.\ $p + q = 2r$ where $p,q , r \in A$.  
We will show that $p$ is an outer term of at most one other non-trivial a.p.  Let $p + q' = 2r'$ 
be another a.p.\ with $q' , r' \in A$ and $(q, r, ) \neq ( q' , r' )$.  

If $r = r'$ then $p + q = 2r = 2r' = p + q'$ so $q = q'$ which is a contradiction and we can assume $r \neq r'$.  

If $q = q'$ then $2r = p + q = p + q' = 2r'$ so $r' = r$ or $r' = r + N/2$ and $p + q = 2r$ and 
$p +q = 2(r + N/2)$.  

Now suppose $r \neq r'$ and $q \neq q'$.  Since $2r - q = p = 2r' - q'$ we have by the $B_3^+$ property, 
\begin{equation*}
\{ r , q' \} \cap \{ r' , q \} \neq \emptyset.
\end{equation*}
The only two possibilities are $r =q$ and $r' = q'$ but in either of these cases we get a trivial 3-term a.p.  
Putting everything together proves 
\begin{lemma}\label{ap}
If $A \subset \integers_N$ is a $B_3^+$-set then the number of 3-term arithmetic progressions in $A$ is at most $3 |A|$. 
\end{lemma}            

Given a fixed element $a \in A$ and a fixed 3-term a.p.\ $c + y = 2b$ in $A$, there are at most $4!$ ways to form an ordered tuple of the form $( ( \{a,c \} , b ) , ( \{ a, b \} , y ) )$.  
The number of ordered tuples counted by $\sum f(n) (f(n) -1)$ when $\{a,c \} \cap \{x,z \} \neq \emptyset $ and $z  =b $ is at most $4! |A| \cdot 3|A|  = 72 |A|^2$.  The first factor of $|A|$ in the expression $4!|A| \cdot 3|A|$ comes from the number of ways to choose the element $a  = x \in \{a,c \} \cap \{x,z \}$.      

Assume now that $z \neq b$.  Recall that we have solved for $c$ to get $c = b - y + z$.  If $b = y$ then $c = z$ which implies $\{a,c \} = \{x,z \}$, a contradiction as the tuples are distinct.  By definition $y \neq z$ so $c = b - y + z$ with $\{b,z \} \in A^{(2)}$ and $\{y\} \cap \{b,z \} =  \emptyset$.  The number of ways to write $c$ in this form is $f(c)$.  Given such a solution $\{b,z \}, y$ counted by $f(c)$, there are two ways to order $b$ and $z$ and $|A|$ ways to choose $a =x$.  The number of ordered tuples we obtain when $\{a,c \} \cap \{x,z \} \neq \emptyset $ and $z \neq b$ is at most $|A| \cdot 2 \sum_{c \in A} f(c)$.  
This completes the analysis in Case 1.  

Before addressing Case 2, the case when $b=y$ and $\{a,c \} \cap \{x,z \} = \emptyset$, some additional notation is needed.  
For $d \in A+A$, define 
\begin{equation*}
S(d) = \left\{ \{a,b\} \in A^{(2)} : a + b = d, \exists \{a' , b' \} \in A^{(2)} ~ \textrm{with} ~ \{a,b\} \cap \{a',b'\} = \emptyset, 
a' + b' =d \right\}.
\end{equation*} 
Let $d_1, d_2, \dots , d_M$ be the integers for which $S(d_i) \neq \emptyset$.  Write $S_{i}^{2}$ for $S(d_i)$ and define  
\begin{equation*}
T_{i}^{1} = \{ a : a \in \{a,b \} ~ \textrm{for some} ~ \{a,b\} \in S_{i}^{2} \}.
\end{equation*}
Let $s_i = | S_{i}^{2} |$ and $d_1, d_2 , \dots , d_m$ be the integers for which $s_i = 2$, and $d_{m+1} , \dots , d_M$ be the integers for 
which $s_i \geq 3$.     
For $1 \leq i \leq M$, we will use the notation 
$S_{i}^{2} = \{ \{ a_{1}^{i} , b_{1}^{i} \} , \{a_{2}^{i}, b_{2}^{i} \}, \dots , \{ a_{s_i}^i , b_{s_i}^i \} \}$.  A simple but important observation is that for any fixed $i \in \{1, \dots , M \}$, any element of $A$ appears in at most one pair in $S_i^2$.   

If $A$ was a $B_3$-set then there would be no $d_i$'s.  This suggests that a $B_3^+$-set or a $B_3^*$-set that is denser than a $B_3$-set should have many $d_i$'s.  The $B_3^+$-set $A_3^+$ constructed in Theorem~\ref{lower bound theorem} has $m \approx  \frac{1}{2}{ |A_3^+| \choose 2}$.  However, if $A_3^+$ is viewed as a subset of $\integers$ then 
$m \approx \frac{1}{4} { |A_3^+ | \choose 2}$ (see Lemma~\ref{not too many} which also holds in $\integers_N$ if $N$ is odd).    

\vspace{1em}
\noindent
\textbf{Case 2:}  $b = y$ and $\{a,c \} \cap \{x,z \} = \emptyset$.

If $b = y$ then $a+c = x+z$.  There are $|A|$ choices for $b = y$ and 
\begin{equation*}
\sum_{i=1}^{M} |S_i^2| ( |S_i^2 | - 1)
\end{equation*}
ways to choose an ordered pair of different sets $\{a,c \}, \{x,z \} \in A^{(2)}$ with $a + c = x+z$ and $\{a,c \} \cap \{x,z \} = \emptyset$.  

\vspace{1em}

Putting Cases 1 and 2 together gives the estimate 
\begin{equation}\label{estimate 0}
\sum f(n)(f(n)-1) \leq  |A| \left( 2 \sum_{c \in A} f(c) + \sum_{i=1}^{M} |S_i^2| ( |S_i^2| - 1)  \right) + 72|A|^2.
\end{equation}
        
Our goal is to find upper bounds on the sums $\sum_{c \in A }f(c)$ and $\sum_{i=1}^{M} |S_i^2| ( |S_i^2| - 1)$. 

\begin{lemma}\label{disjoint}
If $x \in T_i^1 \cap T_j^1$ for some $i \neq j$ then (i) $\max \{ s_i , s_j \} \leq 3$ and (ii) if $s_i = s_j = 3$ then for some 
$x_1, y ,z \in A$ depending on $i$ and $j$ we have $d_j = d_i + N/2$ and 
\begin{equation*}
S_i^2 = \left\{ \{x,x_1 \}, \{y , z \} , \{y + \frac{N}{2} , z + \frac{N}{2} \} \right\}, 
S_j^2 = \left\{ \{x,x_1 + \frac{N}{2} \} , \{y + \frac{N}{2} , z \}, \{y , z + \frac{N}{2} \} \right\}.
\end{equation*}
\end{lemma}
\begin{proof}
If $s_i =2$ and $s_j = 2$ then we are done.  Assume $s_j > 2$ and 
let $S_i^2 = \{ \{a_1^i, b_1^i \}, \dots ,\{a_{s_i}^i , b_{s_i}^i \} \}$ and 
$S_j^2 = \{ \{ a_1^j , b_1^j \}, \dots , \{ a_{s_j}^j, b_{s_j}^j \} \}$.  Without loss of generality, suppose $x = a_{i}^{1}$ and $x = a_{j}^{1}$.  
By definition, $s_i \geq 2$ so we can write $d_i = x + b_{1}^{i} = a_{2}^{i} + b_{2}^{i}$ and 
$d_j = x+ b_{1}^{j} = a_{2}^{j} + b_{2}^{j} = a_{3}^{j} + b_{3}^{j}$.  

Solve for $x$ to get $x = a_2^i + b_2^i - b_1^i = a_2^j + b_2^j - b_1^j$ which can be rewritten as 
\begin{equation}\label{disjoint 1}
a_2^i + b_2^i + b_1^j = a_2^j + b_2^j + b_1^i.
\end{equation}
Since $d_i \neq d_j$, $b_{1}^{i}$ cannot be $b_{1}^{j}$ therefore $b_1^j$ is not on the right hand side of (\ref{disjoint 1}), and $b_1^i$ is not on the left hand side of (\ref{disjoint 1}).  By the $B_3^+$ property,  
\begin{equation*}
\{ a_2^i , b_2^i \} \cap \{a_2^j , b_2^j \} \neq \emptyset.
\end{equation*}
The same argument can be repeated with $a_3^j$ in place of $a_2^j$ and $b_3^j$ in place of $b_2^j$ to get 
\begin{equation*}
\{ a_2^i , b_2^i \} \cap \{a_3^j , b_3^j \} \neq \emptyset.
\end{equation*}
Recall any element of $A$ can occur at most once in the list $a_1^j , b_1^j, a_2^j, b_2^j, \dots a_{s_j}^{j}, b_{s_j}^{j}$ thus $s_j \leq 3$. 
By symmetry, $s_i \leq 3$.  

Now suppose $s_i = s_j = 3$.  Repeating the argument above we have for each $2 \leq k \leq 3$ and $2 \leq l \leq 3$,
\begin{equation*}
| \{a_l^i , b_l^i \} \cap \{ a_k^j , b_k^j \} | =1.
\end{equation*}
This intersection cannot have size 2 since $d_i \neq d_j$.  Without loss of generality, let $y = a_2^i = a_2^j$, $z = b_2^i = a_3^j$, $u = a_3^i = b_2^j$, and $v = b_3^i = b_3^j$.  We represent these equalities between $T_i^1$ and $T_j^1$ using a bipartite graph with 
parts $T_i^1$ and $T_j^1$ where $w \in T_i^1$ is adjacent to $w' \in T_j^1$ if and only if $w = w'$ (see Figure 1). 
\vspace{.5em} 
\begin{center}
\begin{picture}(300,120)
\put(0,30){\circle*{4}}
\put(60,30){\circle*{4}}
\put(120,30){\circle*{4}}
\put(180,30){\circle*{4}}
\put(240,30){\circle*{4}}
\put(300,30){\circle*{4}}

\put(0,90){\circle*{4}}
\put(60,90){\circle*{4}}
\put(120,90){\circle*{4}}
\put(180,90){\circle*{4}}
\put(240,90){\circle*{4}}
\put(300,90){\circle*{4}}

\put(0,30){\line(0,1){60}}
\put(120,30){\line(0,1){60}}
\put(180,30){\line(1,1){60}}
\put(240,30){\line(-1,1){60}}
\put(300,30){\line(0,1){60}}

\put(-18,15){$a_1^j = x$}
\put(-18,97){$a_1^i = x$}
\put(57,15){$b_1^j$}
\put(57,97){$b_1^i$}
\put(102,15){$a_2^j = y$}
\put(102,97){$a_2^i = y$}
\put(162,15){$b_2^j = u$}
\put(162,97){$b_2^i = z$}
\put(222,15){$a_3^j = z$}
\put(222,97){$a_3^i = u$}
\put(283,15){$b_3^j = v$}
\put(283,97){$b_3^i = v$}
\put(137,-5){\Large{$T_j^1$}}
\put(137,110){\Large{$T_i^1$}}

\end{picture}

\vspace{1em}

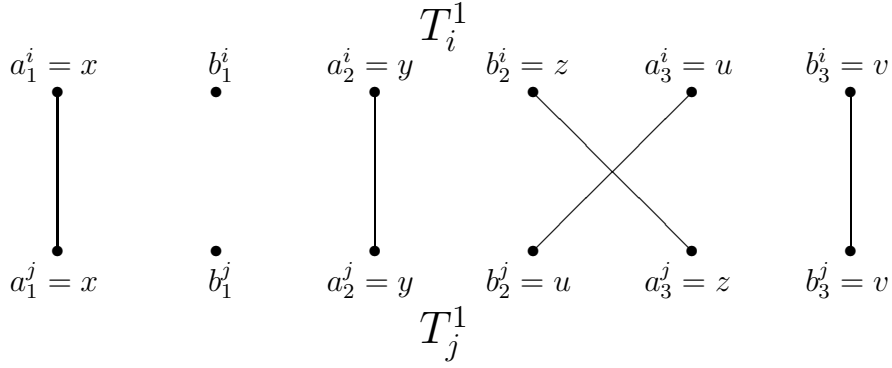
Figure 1 - Equality Graph for Lemma~\ref{disjoint}

\end{center}

\vspace{.5em}

The equalities $d_i = y + z = u+v$ and $d_j = y+u = z + v$ imply $d_i - d_j = z - u$ and $d_i - d_j = u - z$
therefore $2z = 2u$.  If $z = u$ then this is a contradiction since the elements in the list 
$x, b_1^i , y , z, u, v$ are all distinct.  It is in this step that the parity of $N$ plays an important role.  We conclude $u = z + N/2$ and 
\begin{equation*}
d_j = y + u = y + (z + N/2) = y + z + N/2 =d_i + N/2.
\end{equation*}
Let $b_1^i = x_1$ so $b_1^j = x_1 + N/2$.  Since $d_i = y + z = u + v$ and $u = z + N/2$, 
\begin{equation*}
v = y + z - u = y + z - (z + N/2) = y - N/2 = y + N/2.
\end{equation*}
Substituting $u = z+N/2$ and $v = y + N/2$ gives the assertion about the pairs in $S_i^2$ and $S_j^2$ when $s_i = s_j = 3$.   
\end{proof}

\begin{corollary}\label{disjoint c}
If $s_i \geq 4$ then for any $j \neq i$, $T_i^1 \cap T_j^1 = \emptyset$.  Furthermore, any $x \in A$ is in at most two 
$T_i^1$'s with $s_i =3$.  
\end{corollary}
\begin{proof}
The first statement follows immediately from Lemma~\ref{disjoint}.  For the second statement, suppose 
$x \in T_i^1 \cap T_j^1$ with $s_i = s_j = 3$ and $i \neq j$.  By Lemma~\ref{disjoint}, 
$\{ x , x_1 \} \in S_i^2$ and $\{ x , x_1 + N/2 \} \in S_j^2$ for some $x_1 \in A$.  If $x \in T_k^1$ with $k \neq i$ then 
$\{x , x_1 + N/2 \} \in S_k^2$ so $d_j = x + (x_1 + N/2) = d_k$ and $j = k$.  
\end{proof}

\vspace{1em}

\begin{lemma}\label{step 2}
If $A \subset \integers_N$ is a $B_3^+$-set then 
\begin{equation*}
\sum_{c \in A} f(c) \leq  |A|^2 + 7|A|.
\end{equation*}
\end{lemma}
\begin{proof}
For $c \in A$, let 
\begin{equation*}
g_1 (c) = \# \left\{ ( \{x,z \} , y ) \in A^{(2)} \times A : c = x - y + z, c \neq y, \{x,z \} \cap \{y \} = \emptyset \right\}
\end{equation*}
and 
\begin{equation*}
g_2(c) = \# \left\{ ( \{x,z \} , y ) \in A^{(2)} \times A : c = x - y + z, c = y, \{x,z \} \cap \{y \} = \emptyset \right\}.
\end{equation*}
For each $c \in A$, $f(c) = g_1 (c) + g_2 (c)$.  The sum $\sum_{c \in A} g_2 (c)$ is exactly the number of non-trivial 3-term 
a.p.'s in $A$ so by Lemma~\ref{ap}, $\sum_{c \in A} g_1 (c) \leq 3 |A|$.  Estimating $\sum_{c \in A} g_1 (c)$ takes more work.  To compute $g_1 (c)$ with $c \in A$, we first choose an $i$ with $c \in T_i^1$ and then choose one of the pairs $\{x,z \} \in S_i^2 \backslash \{c,y \}$ to obtain a solution $c = x - y +z$ with $c \neq y$ and $\{x,z \} \cap \{y \} = \emptyset$.  

If $c \notin T_1^1 \cup \dots \cup T_M^1$ then the equation $c+y = x+z$ with $c,y,x$, and $z$ all distinct has no solutions in $A$ so $g_1 (c) = 0$.  Assume $c \in T_1^1 \cup \dots \cup T_M^1$.    

\noindent
\textbf{Case 1:}  $c \notin T_1^1 \cup \dots \cup T_m^1$.

By Corollary~\ref{disjoint c} there are two possibilities.  One is that there is a unique $j$ with $c \in T_j^1$ and $s_j \geq 3$ in which case $|S_j^2| \leq \frac{ |A|}{2}$ so $g_1 (c) \leq \frac{ |A|}{2}$.  The other possibility 
is that $c \in T_i^1 \cap T_j^1$ with $s_i = s_j =3$ and $i \neq j$.  In this case $g_1 (c) \leq 4$ because 
we can choose either $i$ or $j$ and then one of the two pairs in $S_i^2$ or $S_j^2$ that does not contain $c$.   

\noindent
\textbf{Case 2:}  $c \in T_1^1 \cup \dots \cup T_m^1$.

By Lemma~\ref{disjoint}, $c$ is not in any $T_j^1$ with $s_j \geq 4$ and $c$ is in at most two $T_j^1$'s with $s_j = 3$.  
There are at most $|A|$~$T_i^1$'s with $c \in T_i^1$ since there are at most $|A|$ pairs $\{c, y \}$ that contain $c$ so $g_1 (c) \leq |A|+4$.    

In all cases, $g_1 (c) \leq |A| +4$ and 
\begin{equation*}
\sum_{c \in A} f(c) = \sum_{c \in A} (g_1(c) + g_2 (c) ) \leq |A|(|A|+4) + 3 |A|
\end{equation*}
which proves the lemma.
\end{proof}

\begin{lemma}\label{new new lemma}
If $g_1(c)$ is the function of Lemma~\ref{step 2} then 
\begin{equation*}
2 \sum_{i=1}^{M} |S_i^2| ( |S_i^2| - 1 ) = \sum_{c \in A} g_1 (c).
\end{equation*}
\end{lemma}
\begin{proof}
Define an edge colored graph $G$ with vertex set $A$, edge set $\cup_{i=1}^{M} S_i^2$, and the color of edge 
$\{ a , b \}$ is $a + b$.  The sum $\sum_{i=1}^{M} |S_i^2| ( |S_i^2| - 1)$ counts ordered pairs 
$( \{c,y \} , \{x,z \} )$ of distinct edges of $G$ where $\{c,y \}$ and $\{x,z \}$ have the same color, i.e.\ 
$c + y = x + z$ and $c,y,x$, and $z$ are all distinct elements of $A$.  
 The sum $\sum_{c \in A} g_1 (c)$ counts each such ordered pair $( \{c,y \}, \{x,z \})$ exactly two times, one contribution 
coming from $g_1(c)$ and the other from $g_1 (y)$.      
\end{proof}  

\vspace{1em}
By Lemma~\ref{new new lemma},
\begin{equation}\label{new new eq}
\sum_{i=1}^{M} |S_i^2| ( |S_i^2| - 1) \leq \frac{1}{2} \sum_{c \in A} f(c).
\end{equation}

Next we use the following version of the Cauchy-Schwarz inequality.

\begin{lemma}[Cauchy-Schwarz]\label{cs}
If $x_1, \dots , x_n$ are real numbers, $t \in \{1, 2, \dots , n-1 \}$, and $\Delta = \frac{1}{t} \sum_{i=1}^{t} x_i 
- \frac{1}{n}  \sum_{i=1}^{n} x_i $ then 
\begin{equation*}
\sum_{i=1}^{n} x_i^2 \geq \frac{1}{n} \left( \sum_{i=1}^{n} x_i \right)^2 + \frac{ t n \Delta^2 }{n - t }.
\end{equation*}
\end{lemma}

A simple counting argument shows $\sum f(n) = { |A| \choose 2}( |A| - 2)$. 
Let $\sum_{c \in A} f(c) = \delta |A|^2$.  
If 
\begin{equation*}
\Delta : = \frac{1}{|A|} \sum_{c \in A} f(c)  - \frac{1}{N} \sum_n f(n) = \delta |A| - \frac{1}{N} \sum_n f(n)
\end{equation*}
then using Ruzsa's bound $|A| = O(N^{1/3})$ and $C_3^+ (N) \leq F_3^+ (N)$,
\begin{equation*}
\Delta = \delta  |A| - \frac{ { |A| \choose 2} ( |A| - 2) }{N } \geq \delta  |A| - C
\end{equation*}
where $C$ is some absolute constant.  By Lemma~\ref{cs},
\begin{eqnarray*}
\sum f(n)^2 & \geq & \frac{ { |A| \choose 2}^2 (|A| -2 )^2 }{ N } + \frac{ |A| \cdot N ( \delta |A| - C)^2 }{ N - |A| } \\
& = & \frac{ { |A| \choose 2}^2 (|A| -2 )^2 }{ N } + \delta^2  |A|^3 \frac{ \left( 1 - \frac{C}{\delta |A|} \right)^2 }{ 1 - \frac{|A|}{N} }.
\end{eqnarray*}  

By (\ref{estimate 0}) and (\ref{new new eq}), 
\begin{eqnarray*}\label{estimate}
\sum f(n)^2 & \leq & \sum f(n) + |A| \left( 2 \sum_{c \in A} f(c) + \sum_{i=1}^{M} |S_i^2| ( |S_i^2|-1) \right) +72 |A|^2 \\
& \leq &  \frac{ |A|^3 }{2} + \frac{ 5 |A| }{2} \sum_{c \in A} f(c) + 72 |A|^2 \\ 
& = & |A|^3 \left( \frac{ 1 + 5 \delta}{2} \right) + 72 |A|^2 .
\end{eqnarray*}
Combining the two estimates on $\sum f(n)^2$ gives the inequality 
\begin{equation}\label{eq 100}
\frac{ { |A| \choose 2}^2 (|A| -2 )^2 }{ N } + \delta^2 |A|^3 \frac{ \left( 1 - \frac{C}{ \delta |A|} \right)^2 }{ 1 - \frac{|A|}{N} } 
\leq  |A|^3 \left( \frac{ 1 + 5 \delta}{2} \right) + 72 |A|^2. 
\end{equation}
If $\delta = 0$ then (\ref{eq 100}) is not valid but we still get 
\begin{equation*}
\frac{  {|A| \choose 2}^2 ( |A| - 2)^2 }{N} \leq \frac{ |A|^3}{2} + 72 |A|^2
\end{equation*}
which implies $|A| \leq (1 + o(1)) (2N)^{1/3}$.  Assume $\delta > 0$.  
In this case (\ref{eq 100}) simplifies to 
\begin{equation}\label{estimate 2}
|A| \leq (1 + o(1) ) \left( 2 + 10 \delta - 4 \delta^2 \right)^{1/3} N^{1/3}.
\end{equation}
At this point we find the maximum of the right hand side of (\ref{estimate 2}) using the fact that $0 \leq \delta \leq 1+ 
\frac{7}{|A|}$ which follows from Lemma~\ref{step 2}.   
For $|A| \geq 28$, the maximum occurs when $\delta =1 + \frac{7}{|A|}$ therefore, after some simplifying, we find 
\begin{equation*}
|A| \leq (1 + o(1)) ( 8 N)^{1/3}.
\end{equation*}

%%%%%%%%%%%%%%%%%%%%%%%%%%%%%%%%%%

\section{Proof of Theorem~\ref{ub 3}(ii)}

The proof of Theorem~\ref{ub 3}(ii) follows along the same lines as the proof of Theorem~\ref{ub 3}(i) and we will use the same 
notation as in the previous section.  The derivation of (\ref{estimate 0}) is very similar except in $\integers$ (or in $\integers_N$ with $N$ odd), there are fewer 3-term a.p.'s in $A$.  Regardless, (\ref{estimate 0}) still holds under the assumption
that $A \subset [N]$ is a $B_3^+$-set or $A \subset \integers_N$ is a $B_3^+$-set with $N$ odd.  

Next we prove a lemma that corresponds to Lemma~\ref{disjoint}.  

\begin{lemma}\label{disjoint v2}
If $x \in T_{i}^{1} \cap T_{j}^{1}$ for distinct $i$ and $j$ then either $s_i = s_j = 2$, or if $s_j > 2$ then $s_i = 2$, $s_j = 3$, and $| T_i^1 \cap T_j^1| \geq 3$.  
\end{lemma}
\begin{proof}
The proof of this lemma is exactly the same as the proof of Lemma~\ref{disjoint} up until the point where we write the equation 
$2z = 2u$.  In $\integers$ (or $\integers_N$ with $N$ odd), this implies $z = u$ which is a contradiction 
since the elements $x, b_1^i, y, z, u, v$ are all distinct.  This allows us to conclude that $T_i^1 \cap T_j^1 = \emptyset$ for 
any $i \neq j$ with $s_i \geq 3$ and $s_j \geq 3$.  

The assertion $|T_i^1 \cap T_j^1 | \geq 3$ can be verified with some easy computations.  Alternatively, one can just ignore 
$a_3^i = u$ and $b_3^i = v$ in Figure 1 to see $|T_i^1 \cap T_j^1 | \geq 3$.  
\end{proof}

\begin{corollary}\label{disjoint c v2}
If $m + 1 \leq i < j \leq M$ then $T_{i}^{1} \cap T_{j}^{1} = \emptyset$.
\end{corollary}
\begin{proof}
If $x \in T_i \cap T_j$ with $i \neq j$ then by Lemma~\ref{disjoint v2} one of $s_i$ or $s_j$ must be equal to 2.
\end{proof}

\vspace{1em}
 
The next lemma has no corresponding lemma from the previous section.   Lemma~\ref{not too many} will be used to estimate $\sum_{c \in A} f(c)$.        

\begin{lemma}\label{not too many}
If $A \subset [N]$ is a $B_3^+$-set or $A \subset \integers_N$ is a $B_3^+$-set and $N$ is odd then for any $a \in A$, the number of distinct $i \in \{1, 2, \dots , m \}$ such that $a \in T_i^1$ is at most $\frac{|A|}{2}$.
\end{lemma}
\begin{proof}
To make the notation simpler, we suppose $a \in T_i^1$ for $1 \leq i \leq k$ and will show $k \leq \frac{|A|}{2}$.  The case when $a \in T_{i_1}^1 \cap \dots \cap T_{i_k}^1$ for some 
sequence $1 \leq i_1 < \dots < i_k \leq m$ is the same.  For this lemma we deviate 
from the notation $S_i^2 = \{ \{a_1^i , b_1^i \} , \dots \{ a_{s_i}^i , b_{s_i}^i \} \}$.  
Write $S_i^2 = \{ \{a,a_i \}, \{b_i , c_i \} \}$ and $a + a_i = b_i + c_i$ where $ 1 \leq i \leq k$ and for fixed $i$, $a, a_i , b_i$, and $c_i$ are all distinct.  Observe $a_1, \dots , a_k$ are all distinct since the sums $a+a_i$ are all distinct.  For $1 \leq i \leq k$, $a = b_i + c_i - a_i$ hence
\begin{equation*}
b_i + c_i + a_j = b_j + c_j + a_i
\end{equation*}
for any $1 \leq i , j \leq k$.  These two sums must intersect and they cannot intersect at $a_j$ or $a_i$, unless $i =j$, so for $2 \leq j \leq k$,
\begin{equation*}
\{b_1, c_1 \} \cap \{b_j, c_j \} \neq \emptyset.
\end{equation*}
Let $2 \leq j \leq l$ be the indices for which the sums intersect at $b_1$.  Let $l+1 \leq j \leq k$ be the indices for which the sums intersect at $c_1$ and let $b = b_1$ and $c = c_1$.  We have the $k$ equations
\begin{eqnarray*}
a + a_1 & = & b + c, \\
a + a_2 & = & b + c_2, \\
~~ \vdots ~~ & ~ & ~~ \vdots ~~ \\
a + a_l & = & b + c_l, \\
a + a_{l+1} & = & b_{l+1} + c, \\
~~ \vdots ~~ & ~ & ~~ \vdots ~~ \\
a + a_k & = & b_k + c.
\end{eqnarray*}

We will show $a_1, \dots , a_k, c_1, \dots , c_l, b_{l+1}, \dots , b_{k}$ are all distinct which implies $2k \leq |A|$.  

Suppose $a_i = b_j$ for some $2 \leq i \leq l$ and $l+1 \leq j \leq k$.  Then $a + b_j = a+a_i = b + c_i$ but $a = b_j + c - a_j$ so 
$b + c_i = a+ b_j = 2b_j + c - a_j$ which implies $2b_j + c = b + c_i + a_j$.  The elements $a_j, b_j$, and $c$ are all distinct so these sums cannot intersect at $a_j$.  Similarly they cannot intersect at $c$.  The only remaining possibility is $b_j = c_i$ but then $a_i = b_j = c_i$, a contradiction.  
We conclude that $a_i$ and $b_j$ are distinct for $2 \leq i \leq l$, $l+1 \leq j \leq k$.  A similar argument shows $a_j$ and $c_i$ are distinct for $l+1 \leq j \leq k$ and $2 \leq i \leq l$.  

Suppose now that $a_i = c_{i'}$ for some $2 \leq i \neq i' \leq l$.  Then $b + c_i = a + a_i = a + c_{i'} = a + (a + a_{i'} - b)$ so that 
$2b + c_i = 2a + a_{i'}$.  Since $2 \leq i' \leq l$, these sums cannot intersect at $b$ and also they cannot intersect at $a$.  If $c_i = a_{i'}$ then $a = b$ which is impossible therefore the equation $2b + c_i = 2a + a_{i'}$ contradicts the $B_3^+$ property.  
Note that $2b = 2a$ need not imply $a = b$ if $A \subset \integers_N$ with $N$ even.  We conclude 
$a_i \neq c_{i'}$ for each $2 \leq i \neq i' \leq l$.  Similarly $a_j \neq b_{j'}$ for $l+1 \leq j \neq j' \leq k$.  

The previous two paragraphs imply 
\begin{equation*}
\{ a_1, a_2 , \dots , a_k \} \cap \{c_2,c_3, \dots , c_l, b_{l+1}, b_{l+2}, \dots , b_k \} = \emptyset. 
\end{equation*}
To finish the proof we show $\{c_2, c_3, \dots , c_l \} \cap \{b_{l+1}, b_{l+2}, \dots , b_k \} = \emptyset$.  
Suppose $c_i = b_j$ for some $2 \leq i \leq l$ and $l+1 \leq j \leq k$.  Then 
\begin{equation*} 
a + a_i = b + c_i = b + b_j = b + (a + a_j - c) = b + a + a_j - (a + a_1 - b) = a_j + 2b - a_1
\end{equation*}
which implies $a + a_i + a_1 = a_j + 2b$.  Since $ i < l+1 \leq j$, these sums cannot intersect at $a_j$.  They cannot intersect at $b$ either since $a,a_i,b$, and $c_i$ are all distinct whenever $1 \leq i \leq l$.  This is a contradiction therefore $c_i \neq b_j$ for all $2 \leq i \leq l$ and $l+1 \leq j \leq k$.  
\end{proof}

\begin{lemma}\label{step 2 v2}
If $A \subset [N]$ is a $B_3^+$-set then 
\begin{equation*}
\sum_{c \in A} f(c) \leq \frac{ |A|^2}{2} + 3|A|.
\end{equation*}
\end{lemma}
\begin{proof}
As before we write $f$ as a sum of the simpler functions $g_1$ and $g_2$.  Recall for $c \in A$,  
\begin{equation*}
g_1 (c) = \# \left\{ ( \{x,z \} , y ) \in A^{(2)} \times A : c = x - y + z, c \neq y, \{x,z \} \cap \{y \} = \emptyset \right\}
\end{equation*}
and 
\begin{equation*}
g_2(c) = \# \left\{ ( \{x,z \} , y ) \in A^{(2)} \times A : c = x - y + z, c = y, \{x,z \} \cap \{y \} = \emptyset \right\}.
\end{equation*}
Again for each $c \in A$, $f(c) = g_1 (c) + g_2 (c)$.  The sum $\sum_{c \in A} g_2 (c)$ is exactly the number of non-trivial 3-term a.p.'s in $A$ and by Lemma~\ref{ap} this is at most $3|A|$. 

If $c \notin T_1^1 \cup \dots \cup T_M^1$ then the equation $c+y = x+z$ with $c,y,x$, and $z$ all distinct has no solutions in $A$ so $g_1 (c) = 0$.  Assume $c \in T_1^1 \cup \dots \cup T_M^1$.  

\noindent
\textbf{Case 1:}  $c \notin T_1^1 \cup \dots \cup T_m^1$.

By Corollary~\ref{disjoint c v2}, there exists a unique $j$ with $c \in T_j^1$ and $m+1 \leq j \leq M$.  For such a $j$ we have $|S_j^2| \leq \frac{ |A|}{2}$ again by Corollary~\ref{disjoint c v2}.  There is a unique 
pair in $S_j^2$ that contains $c$ so $y$ is determined and there are at most $\frac{ |A| }{2}$ choices for the pair 
$\{x,z \} \in S_j^2 \backslash \{c,y \}$ so 
$g_1(c) \leq \frac{ |A|}{2}$.  

\noindent
\textbf{Case 2:}  $c \in T_1^1 \cup \dots \cup T_m^1$.

First assume $c \notin T_{m+1}^1 \cup \dots \cup T_M^1$.  A solution to $c+y = x+z$ with $c,y,x$, and $z$ all distinct corresponds to a choice of an $S_i^2$ with $1 \leq i \leq m$ and $c \in T_i^1$.  By Lemma~\ref{not too many}, $c$ is in at most $\frac{ |A|}{2}$ $T_i^1$'s and so $g_1 (c) \leq \frac{ |A|}{2}$.  

Lastly suppose $c \in T_{m+1}^1 \cup \dots \cup T_M^1$.  There exists a unique $j$ with $c \in T_j^1$ and $m+1 \leq j \leq M$.  Furthermore for this $j$, $|T_j^1 | = 6$ by Lemma~\ref{disjoint v2}.  If $c \in T_i^1$ with $1 \leq i \leq m$ then, again by Lemma~\ref{disjoint v2}, $|T_i^1 \cap T_j^1 | \geq 3$.  There are ${6 \choose 3}$ 3-subsets of $T_j^1$ and given such a 3-subset there are ${3 \choose 1}$ ways to pair up an element in the 3-subset with $c$ in $S_i^2$.  This implies $c$ is in at most $3 {6 \choose 3}$~$S_i^2$'s with $1 \leq i \leq m$ so 
$g_1 (c) \leq 2 +  3{6 \choose 3} \leq \frac{ |A| }{2}$.  The 2 comes from choosing one of the two pairs in $S_j^2 \backslash \{c,y \}$.    
\end{proof}   

\vspace{1em}

The rest of the proof of Theorem~\ref{ub 3}(ii) is almost identical to that of Theorem~\ref{ub 3}(i).  
If $\sum_{c \in A} f(c) = \delta |A|^2$ then by  (\ref{estimate 0}) and (\ref{new new eq}),
\begin{equation*}
\sum f(n)^2 \leq |A|^3 \left( \frac{1 + 5 \delta}{2} \right) + O( |A|^2).
\end{equation*}

We use the same version of the Cauchy-Schwarz inequality to get 
\begin{equation}\label{another new eq}
\frac{   { |A| \choose 2}^2 ( |A| - 2)^2 }{3N} + \delta^2 |A|^3 \frac{  \left( 1 - \frac{C}{\delta |A|} \right) }{1 - \frac{ |A|}{3N} } 
\leq |A|^3  \left( \frac{1 + 5 \delta}{2} \right) + O( |A|^2).
\end{equation}
If $\delta = 0$ then 
\begin{equation*}
\frac{  { |A| \choose 2}^2 (|A| - 2)^2 }{3N} \leq \frac{ |A|^3}{2} + O (|A|^2)
\end{equation*}
which implies $|A| \leq (1 + o(1))(6N)^{1/3}$.  Assume $\delta > 0$.  Then (\ref{another new eq}) simplifies to 
\begin{equation*}
|A| \leq (1 + o(1)) ( 6 + 30 \delta - 12 \delta^2 )^{1/3} N^{1/3}.
\end{equation*}
By Lemma~\ref{step 2 v2}, $0 \leq \delta \leq \frac{1}{2} +  \frac{3}{|A|}$.  The maximum occurs when 
$\delta = 1 + \frac{3}{ |A|}$ and we get 
\begin{equation*}
|A| \leq (1 + o(1))(18N)^{1/3}.
\end{equation*}

If we were working in $\integers_N$ with $N$ odd then in (\ref{another new eq}), the $3N$ can be replaced by $N$ and some simple calculations show that we get Theorem~\ref{ub 3}(i) in the odd case.  We actually obtain the upper bound 
$|A| \leq (1 + o(1))(6 N)^{1/3}$ when $A \subset \integers_N$ is a $B_3^+$-set and $N$ is odd.

%%%%%%%%%%%%%%%%%%%%%%%%%%

\section{Proof of Theorem~\ref{ub 3}(iii)}

%%%%%%%%%%%%%%%%%%%%%%%%%

Let $A \subset [N]$ be a $B_4^+$-set.  For $n \in [-2N,2N]$, define
\begin{eqnarray*}
f(n) = \# \{ ( \{a_1,a_2 \} , \{b_1, b_2 \} ) & \in &  A^{ (2) } \times A^{(2)} : a_1 + a_2 - b_1 - b_2 = n, \\
& ~ & \{a_1,a_2 \} \cap \{b_1, b_2 \} = \emptyset \}. 
\end{eqnarray*} 

\begin{lemma}\label{l 0}
If $A \subset [N]$ is a $B_4^+$-set then $A$ is a $B_2$-set.
\end{lemma}
\begin{proof}
Suppose $a+b = c+d$ with $a,b,c,d \in A$.  If $\{a,b \} \cap \{c,d \} = \emptyset$ then the equation $2(a+b) = 2(c+d)$ contradicts the $B_4^+$ property so $\{a, b \} \cap \{c,d \} \neq \emptyset$.  Since $a+b = c+d$ and $\{a,b\} \cap \{c,d \} \neq \emptyset$, we have $\{a,b \} = \{ c , d \}$.
\end{proof}

\begin{lemma}\label{l 1}
If $A \subset [N]$ is a $B_4^+$-set then for any $n$, $f(n) \leq 2|A|$.
\end{lemma}
\begin{proof}
Suppose $f(n) \geq 1$.  Fix a tuple $( \{a_1,a_2 \} , \{b_1,b_2 \})$ counted by $f(n)$.  Let $( \{c_1,c_2 \} , \{d_1,d_2 \})$ be another tuple counted by $f(n)$, not necessarily different from 

\noindent
$( \{a_1,a_2 \} , \{b_1,b_2 \})$.  Then $a_1 + a_2 - b_1 - b_2 = c_1 + c_2 - d_1 - d_2$ so 
\begin{equation}\label{eq 1}
a_1 + a_2 + d_1 + d_2 = c_1 + c_2 + b_1 + b_2.
\end{equation}
By the $B_4^+$ property, $\{a_1,a_2, d_1 ,d_2 \} \cap \{ c_1 , c_2 , b_1 , b_2 \} \neq \emptyset$.  In order for this intersection to be non-empty, it must be the case that $\{a_1,a_2 \} \cap \{c_1,c_2 \} \neq \emptyset$ or $\{b_1,b_2 \} \cap \{d_1,d_2 \} \neq \emptyset$.  

\noindent
\textbf{Case 1:}  $\{a_1,a_2 \} \cap \{c_1,c_2 \} \neq \emptyset$.

Assume $a_1 = c_1$.  There are 
at most $|A|$ choices for $c_2$ so we fix one.  The equality $a_1 = c_1$ and (\ref{eq 1}) imply 
\begin{equation}\label{eq 2}
d_1 + d_2 = b_1 + b_2 + c_2 - a_2.
\end{equation}
The right hand side of (\ref{eq 2}) is determined so that by Lemma~\ref{l 0} there is at most one pair $\{d_1,d_2 \}$ such that (\ref{eq 2}) holds.   

\noindent
\textbf{Case 2:}  $\{a_1,a_2 \} \cap \{c_1,c_2 \} = \emptyset$ and $\{b_1,b_2 \} \cap \{d_1,d_2 \} \neq \emptyset$.

Again there is no loss in assuming
$b_1 = d_1$.  There are at most $|A|$ choices for $d_2$ so fix one.  The equality $b_1 = d_1$ and (\ref{eq 1}) imply 
\begin{equation}\label{eq 3}
c_1 + c_2 = a_1 + a_2 - b_2 + d_2.
\end{equation}
The right hand side of (\ref{eq 3}) is determined and there is at most one pair $\{c_1,c_2 \}$ satisfying (\ref{eq 3}) as before.

Putting the two possibilities together we get at most $2|A|$ solutions $( \{c_1,c_2 \} , \{d_1 , d_2 \})$ and we have also 
accounted for the solution $( \{a_1,a_2 \} , \{b_1 , b_2 \})$ in our count so $f(n) \leq 2 |A|$.
\end{proof}

\begin{lemma}\label{l 2}
If $A \subset [N]$ is a $B_4^+$-set then 
\begin{equation}\label{eq 4}
\sum f(n) ( f(n) - 1) \leq 2|A| \sum_{n \in A - A} f(n).
\end{equation}
\end{lemma}
\begin{proof}
The left hand side of (\ref{eq 4}) counts the number of ordered tuples 
\begin{equation*}
\left(  ( \{a_1,a_2 \} , \{b_1,b_2 \} ) , ( \{c_1,c_2 \} , \{d_1,d_2 \} ) \right)
\end{equation*}
such that $( \{a_1,a_2 \} , \{b_1,b_2 \} ) \neq ( \{c_1,c_2 \} , \{d_1,d_2 \} )$ and both tuples are counted by $f(n)$.  Equation (\ref{eq 1}) holds for these tuples and as before we consider two cases.  

\noindent
\textbf{Case 1:} $\{a_1,a_2 \} \cap \{c_1 , c_2 \} \neq \emptyset$.

Assume $a_1 = c_1$ so that 
$a_2 - c_2 = b_1 + b_2 - d_1 - d_2$.  

If $\{b_1, b_2 \} \cap \{d_1,d_2 \} \neq \emptyset$, say $b_1 = d_1$, then 
$a_2  - c_2 = b_2 - d_2$.  We can rewrite this equation as $a_2 + d_2 = b_2 + c_2$ so that $\{a_2, d_2 \} = \{b_2,c_2 \}$.  Since 
$\{a_1,a_2 \} \cap \{b_1,b_2 \} = \emptyset$, it must be the case that $a_2 = c_2$ and $d_2 = b_2$ but this contradicts the fact that 
the tuples are distinct.  We conclude $\{b_1, b_2 \} \cap \{d_1,d_2 \} = \emptyset$. 

There are $|A|$ choices for the element $a_1 = c_1$ and we fix one.  Since $a_2-c_2 = b_1 + b_2 - d_1 - d_2$ and 
$\{b_1, b_2 \} \cap \{d_1,d_2 \} = \emptyset$, there are $f(a_2 - c_2)$ ways to choose $\{b_1,b_2 \}$ and $\{d_1,d_2 \}$.  Also observe that each $n \in A - A$ with $n \neq 0$ has a unique representation as 
$n = a_2 - c_2$ with $a_2, c_2 \in A$.  This follows from the fact that $A$ is a $B_2$-set.       

\noindent
\textbf{Case 2:}  $\{a_1,a_2 \} \cap \{c_1 , c_2 \} = \emptyset$ and $\{b_1,b_2 \} \cap \{d_1,d_2 \} \neq \emptyset$.
 
The argument in this case is essentially the same as that of Case 1.  

Putting the two cases together gives the lemma.

\end{proof} 

\vspace{1em}

Observe $\sum f(n) = { |A| \choose 2} { |A| - 2 \choose 2}$.   
Using Cauchy-Schwarz, Lemma~\ref{l 2}, and Lemma~\ref{l 1}, 
\begin{eqnarray*}
\frac{  \left(  { |A| \choose 2} {|A| - 2 \choose 2} \right)^2 }{ 4N} & \leq & \sum f(n)^2 \leq { |A| \choose 2} { |A| - 2 \choose 2} 
+ 2 |A| \sum_{n \in A - A} f(n) \\
& \leq & \frac{ |A|^4 }{4} + 2 |A||A - A| \cdot 2|A| \\
& \leq & \frac{ |A|^4 }{4} + 4 |A|^4 = \frac{17 |A|^4 }{4}.
\end{eqnarray*}  
After rearranging we get
\begin{equation*}
|A| \leq (1 + o(1)) ( 16 \cdot 17 N)^{1/4} = (1 + o(1)) ( 272 N )^{1/4}.
\end{equation*}

%%%%%%%%%%%%%%%%%%%%%%%%%%%%%%

\section{Proof of Theorem~\ref{upper bounds}}

%%%%%%%%%%%%%%%%%%%%%%%%%%%%

\begin{lemma}\label{1.1 lemma 1}
Let $A$ be a $B_k^+$-set with $k \geq 4$.  If $k = 2l$ then there is a subset $A' \subset A$ such that 
$A'$ is a $B_l^+$-set and $|A'| \geq |A| - 2l$.  If $k = 2l+1$ then there is a subset $A' \subset A$ such that 
$|A'| \geq |A| - 2k$ and $A'$ is either a $B_l^+$-set or a $B_{l+1}^+$-set.
\end{lemma}
\begin{proof}
Suppose $k = 2l$ with $l \geq 2$.  If $A$ is not a $B_l^+$-set then there is a set of $2l$, not necessarily distinct elements
$a_1 , \dots , a_{2l} \in A$, such that 
\begin{equation*}
a_1 + \dots + a_l = a_{l+1} + \dots + a_{2l}
\end{equation*}
and $\{a_1, \dots , a_l \} \cap \{ a_{l+1} , \dots , a_{2l} \} = \emptyset$.  Let $A' = A \backslash \{a_1, a_2, \dots , a_{2l} \}$.  
If $A'$ is not a $B_l^+$-set then there is another set of $2l$ elements of $A'$, say 
$b_1, \dots , b_{2l}$, such that 
\begin{equation*}
b_1 + \dots + b_l = b_{l+1} + \dots + b_{2l}
\end{equation*}
and $\{ b_1, \dots , b_l \} \cap \{ b_{l+1} , \dots , b_{2l} \} = \emptyset$.  Adding these two equations together gives 
\begin{equation*}
a_1 + \dots +a_l + b_1 + \dots + b_l  = a_{l+1} + \dots + a_{2l} + b_{l+1} + \dots + b_{2l}
\end{equation*}
with $\{a_1, \dots , a_l , b_1 , \dots , b_l \} \cap \{ a_{l+1} , \dots , a_{2l} , b_{l+1} , \dots , b_{2l} \} = \emptyset$, a contradiction. 

The case when $k = 2l+1 \geq 5$ can be handled in a similar way.

\end{proof}

\vspace{1em}

It is easy to modify the proof of Lemma~\ref{1.1 lemma 1} to obtain a version for $B_k^*$-sets. 
\begin{lemma}\label{1.1 lemma 2}
Let $A$ be a $B_k^*$-set with $k \geq 4$.  If $k = 2l$ then there is a subset $A' \subset A$ such that 
$A'$ is a $B_l^*$-set and $|A'| \geq |A| - 2l$.  If $k = 2l+1$ then there is a subset $A' \subset A$ such that 
$|A'| \geq |A| - 2k$ and $A'$ is either a $B_l^*$-set or a $B_{l+1}^*$-set.
\end{lemma}

For $A \subset [N]$ and $j \geq 2$, let 
\begin{equation*}
\sigma_j (n) = \# \left\{  (a_1, \dots , a_j ) \in A^j : a_1 + \dots + a_j = n \right\}.
\end{equation*}
Let $e(x) = e^{2 \pi i x}$ and $f(t) = \displaystyle\sum_{a \in A} e(at)$.  
For any $j \geq 1$, $f(t)^j = \sum \sigma_j (n) e(nt)$ so by Parseval's Identity, $\sum \sigma_j (n)^2 = \int_0^1 |f(t)|^{2j} dt$.  The next lemma is (5.9) of \cite{R}.  

\begin{lemma}\label{1.1 lemma 3}
If $A \subset [N]$ is a $B_k^*$-set then 
\begin{equation}\label{1.1 eq 1}
\sum \sigma_k (n)^2 \leq (1 + o(1) ) k^2 |A| \sum \sigma_{k - 1} (n)^2.
\end{equation}
\end{lemma}

In \cite{R}, Ruzsa estimates the right hand side of (\ref{1.1 eq 1}) using H\"{o}lder's Inequality and shows
\begin{equation*}
\sum \sigma_{k-1}(n)^2 \leq \left( \sum \sigma_k (n) \right)^{ \frac{k-2}{k-1} } |A|^{  \frac{1}{k-1} }.
\end{equation*}
Our next lemma uses H\"{o}lder's Inequality in a different way.  

\begin{lemma}\label{1.1 lemma 4}
Let $A \subset [N]$ be a $B_k^*$-set.  If $k \geq 4$ is even then 
\begin{equation*}
\sum \sigma_k (n)^2 \leq (1 + o(1)) k^k |A|^{k/2} \sum \sigma_{k/2}(n)^2.
\end{equation*}
If $k = 2l+1 \geq 5$ then 
\begin{equation*}
\sum \sigma_k (n)^2 \leq (1 + o(1)) \max \left\{ k^{k+1} |A|^{ l+1 } \sum \sigma_l (n)^2 , k^{k-1} |A|^{ l } 
\sum \sigma_{l+1} (n)^2 \right\}.
\end{equation*}
\end{lemma}  
\begin{proof}
First assume that $k = 2l \geq 4$.  By Lemma~\ref{1.1 lemma 2} we may assume that $A$ is a $B_l^*$-set otherwise 
we pass to a subset of $A$ that is a $B_l^*$ set and has at least $|A| - 2k$ elements.  
Applying H\"{o}lder's Inequality with $p = \frac{k}{k-2}$ and $q = \frac{k}{2}$, 
\begin{eqnarray*}
\sum \sigma_{k-1}(n)^2 & = & \int_0^1 |f(t)|^{2(k-1)} dt = \int_0^1 |f(t)|^{ \frac{2k}{p} } |f(t)|^{ \frac{2l}{q} } dt \\ 
& \leq &  \left( \int_0^1 |f(t)|^{2k} dt \right)^{1/p}  \left( \int_0^1 |f(t)|^{2l} dt \right)^{1/q} \\
& = & \left( \sum \sigma_k (n)^2 \right)^{ (k-2)/k } \left( \sum \sigma_l (n)^2 \right)^{ 2/k }.
\end{eqnarray*}
Substituting this estimate into (\ref{1.1 eq 1}) and solving for $\sum \sigma_k (n)^2$ gives the first part of the lemma.   

Now assume $k = 2l +1 \geq 5$.  Again by Lemma~\ref{1.1 lemma 2} we can assume that $A$ is either a $B_l^*$-set or a $B_{l+1}^*$-set.  

Suppose $A$ is a $B_l^*$-set.  Applying H\"{o}lder's Inequality with 
$p = \frac{k+1}{k-1}$ and $q = \frac{k+1}{2}$ we get 
\begin{equation*}
\sum \sigma_{k-1} (n)^2 \leq \left( \sum \sigma_k (n)^2 \right)^{ \frac{k-1}{k+1} } \left( \sum \sigma_l (n)^2 \right)^{ \frac{2}{k+1} }.
\end{equation*}
This inequality and (\ref{1.1 eq 1}) imply 
\begin{equation*}
\sum \sigma_k (n)^2 \leq (1 + o(1)) k^{k+1} |A|^{ \frac{k+1}{2} } \sum \sigma_l (n)^2.
\end{equation*}
If $A$ is a $B_{l+1}^*$-set instead then apply H\"{o}lder's Inequality with $p = \frac{l}{l-1}$ and 
$q = \frac{1}{l}$ and proceed as above.  It is in this step that we must assume $k =2l+1 \geq 5$ otherwise if $k =3$, then $l=1$ and $p$ is not defined.   
\end{proof}

\vspace{1em}

For $k \geq 2$ let $c_k^+$ be the smallest constant such that for any $B_k^+$-set $A$, 
\begin{equation*}
\sum \sigma_k (n)^2 \leq (1 + o(1)) c_k^+ |A|^k.
\end{equation*}
Define $c_k^*$ similarly.  The techniques of \cite{R} can be used to show that $c_k^* \leq k^{2k}$ so 
$c_k^+$ and $c_k^*$ are well defined.  Observe that for any $k \geq 2$, $c_k^+ \leq c_k^*$.  Using Lemma~\ref{1.1 lemma 4}, it is not difficult to show that for even $k \geq 4$, 
\begin{equation}\label{1.2 even}
c_k^+ \leq k^k c_{k/2}^+ ~ \mbox{and} ~ c_k^* \leq k^k c_{k/2}^*,
\end{equation}
and for odd $k = 2l+1 \geq 5$, 
\begin{equation}\label{1.2 odd}
c_k^+ \leq \max \left\{ k^{k+1} c_l^+ , k^{k-1} c_{l+1}^+ \right\} ~ \mbox{and} ~ c_k^* \leq \max \left\{ k^{k+1} c_l^* , k^{k-1} c_{l+1}^* \right\}.
\end{equation}

\begin{lemma}\label{1.1 lemma 5}
Let $A \subset [N]$ be a $B_k^+$-set.  If $k \geq 4$ is even then 
\begin{equation}\label{even star rec}
|A| \leq (1 + o(1)) \left( k^{k + 1} c_{k/2}^+ N \right)^{1/k}.
\end{equation}
If $k = 2l  +1 \geq 5$ then 
\begin{equation}\label{odd star rec}
|A| \leq (1 + o(1)) \left( k^k \cdot \max \{k^2 c_l^+ , c_{l+1}^+ \} N \right)^{1/k}.
\end{equation}
The same inequalities hold under the assumption that $A \subset [N]$ is a $B_k^*$-set provided the $c_k^+$'s are 
replaced with $c_k^*$'s.
\end{lemma}
\begin{proof}
By Cauchy-Schwarz,
\begin{equation}\label{1.1 CS}
\frac{ |A|^{2k} }{kN} \leq \sum \sigma_k (n)^2
\end{equation}
for any $k \geq 2$.  

First suppose $k \geq 4$ is even.  By (\ref{1.1 CS}) and Lemma~\ref{1.1 lemma 4},
\begin{equation*}
\frac{ |A|^{2k} }{kN} \leq \sum \sigma_k (n)^2  \leq   (1 + o(1)) k^k |A|^{k/2} \sum \sigma_{k/2}(n)^2 
 \leq   (1 + o(1)) k^k c_{k/2}^+ |A|^k.
\end{equation*}
Solving this inequality for $|A|$ proves (\ref{even star rec}).  

Now suppose $k = 2l+1 \geq 5$.  By (\ref{1.1 CS}) and Lemma~\ref{1.1 lemma 4},  
\begin{eqnarray*}
\frac{ |A|^{2k} }{kN} & \leq & \sum \sigma_k (n)^2 \leq (1 + o(1)) \max \left\{  k^{k+1} c_l^+ |A|^k , k^{k-1} c_{l+1}^+ |A|^k \right\} \\
& = & (1 + o(1)) |A|^k k^{k-1} \max \{ k^2 c_l^+ , c_{l+1}^+ \}.
\end{eqnarray*}
\end{proof}

\vspace{1em}

Lemma~\ref{1.1 lemma 5} shows that we can obtain upper bounds on $B_k^+$-sets and $B_k^*$-sets recursively.  To start the recursion we need estimates on $c_2^+$, $c_2^*$, $c_3^+$, and $c_3^*$.  

\begin{lemma}\label{1.1 lemma 6}
If $A$ is a $B_2^*$-set then 
\begin{equation*}
\sum \sigma_2 (n)^2 \leq 2 |A|^2 + 32 |A|
\end{equation*}
and therefore $c_2^* \leq 2$.  
\end{lemma}
\begin{proof}
Let $\delta (n) = \# \{ (a_1, a_2) \in A^2 : a_1 - a_2 = n \}$ and observe $\sum \sigma_2 (n)^2 = \sum \delta(n)^2$.
In \cite{R} it is shown that $\delta(n) \leq 2$ for any $n \neq 0$ and $\delta (n) =2$ for at most $8|A|$ integers $n$.  We conclude 
\begin{equation*}
\sum \delta (n)^2 \leq \delta (0)^2 + 8 |A| \cdot 4 + |A - A| \leq 2 |A|^2 + 32 |A|.
\end{equation*}
\end{proof}

\begin{lemma}\label{1.1 lemma 7}
If $A \subset [N]$ is a $B_3^+$-set then 
\begin{equation*}
\sum \sigma_3 (n)^2 \leq (1 + o(1)) 18 |A|^3
\end{equation*}
therefore $c_3^+ \leq 18$.
\end{lemma}
\begin{proof}
Let $A \subset [N]$ be a $B_3^+$-set and let 
\begin{equation*}
r_2 (n) = \# \left\{ \{a,b \} \in A^{(2)} : a + b = n \right\}.
\end{equation*}
Define $2 \cdot A := \{ 2a : a \in A \}$.  For $n \in 2 \cdot A$, $\sigma_2 (n) = 2 r_2 (n) +1$ and 
$\sigma_2 (n) = 2 r_2 (n)$ otherwise.  The sum $\sum_{ n \in 2 \cdot A} r_2(n)$  counts the number of 3-term a.p.'s in 
$A$ so by Lemma~\ref{ap},
\begin{eqnarray*}
\sum \sigma_2 (n)^2 & = & 4 \sum r_2 (n)^2 + 4 \sum_{n \in 2 \cdot A} r_2(n) + |2 \cdot A| \\
& \leq & 4 \sum r_2 (n)^2 + 4 \cdot 3 |A| + |A| = 4 \sum r_2 (n)^2 + 13 |A|.
\end{eqnarray*}
Using the notation and results of Section 3 and the inequality $x^2 \leq 2x(x-1)$ where $x \geq 2$, 
\begin{equation*}
\sum r_2 (n)^2  =  \sum_{i=1}^{M} |S_i^2|^2 \leq 2 \sum_{i=1}^{M} |S_i^2| ( |S_i^2| -1) 
 \leq  \sum_{c \in A} f(c) \leq \frac{ |A|^2}{2} + 3|A|.
\end{equation*}
Using (\ref{1.1 eq 1})
\begin{eqnarray*}
\sum \sigma_3 (n)^2 & \leq & (1 + o(1)) 3^2 |A| \sum \sigma_2 (n)^2 \leq (1 + o(1))9 |A| ( 4 \sum r_2 (n)^2 + 13 |A| ) \\
& \leq & (1 + o(1)) 9 |A| ( 2 |A|^2 + 25 |A| ) \leq (1 + o(1))(18 |A|^3 + 225 |A|^2).
\end{eqnarray*} 
\end{proof}

\vspace{1em}

\begin{lemma}\label{1.1 lemma 7.5}
If $A \subset [N]$ is a $B_3^*$-set then 
\begin{equation*}
\sum \sigma_3 (n)^2 \leq (1 + o(1)) 54 |A|^3
\end{equation*}
therefore $c_3^* \leq 54$.
\end{lemma}
\begin{proof}
Let $A\subset [N]$ be a $B_3^*$-set.  The idea of the proof is motivated by the same arguments that were used for $B_3^+$-sets.  
 For $d \in A + A$, let 
\begin{equation*}
P^2 (d) = \{ \{a,b \} \in A^{(2)} : a + b = d \}.
\end{equation*}
Define $m_0 = 0$ and for $1 \leq j \leq 4$, let $d_{m_{j-1} +1 } , d_{m_{j-1} +2 } , \dots , d_{m_j}$ be the integers 
for which $|P^2 (d_i)| = j$.  Let $d_{m_4 +1} , d_{m_4 +2}, \dots , d_M$ be the integers for which 
$|P^2 (d_i)| \geq 5$.  As before, write $P_i^2$ for $P^2 (d_i)$, $p_i$ for $|P_i^2|$, and for $1 \leq i \leq M$ let 
\begin{equation*}
Q_i^1 = \{ a : a \in \{a,b\} ~ \mbox{for some} ~ \{a,b \} \in P_i^2 \}.
\end{equation*}
We will use the notation $P_i^2 = \{ \{a_1^i , b_1^i \}, \dots , \{a_{p_i}^i , b_{p_i}^i \} \}$.  A difference between the $P_i^2$'s of this section and the $S_i^2$'s of earlier sections is that we allow for a $P_i^2$ to be one pair.  

\begin{lemma}\label{1.1 lemma 8}
If $x \in Q_i^1 \cap Q_j^1$ for some $i \neq j$ and $p_i \geq 3$ and $p_j \geq 3$ then $p_i+p_j \leq 7$.  
\end{lemma}
\begin{proof}
Without loss of generality, assume $x = a_1^i$ and $x=a_1^j$ where 
\begin{equation*}
P_i^2 = \{ \{a_1^i , b_1^i \},  \{a_2^i , b_2^i \}, \dots , \{ a_{p_i}^i , b_{p_i}^i \} \} ~~ \textrm{and} ~~ 
P_j^2 = \{ \{a_1^j , b_1^j \}, \{a_2^j , b_2^j \}, \dots , \{ a_{p_j}^j , b_{p_j}^j \} \}.
\end{equation*}
For $2 \leq l \leq p_i$ we have $d_i = x + b_1^i = a_l^i + b_l^i$ and similarly for $2 \leq k \leq p_j$ we have $d_j = x + b_1^j = a_k^j + b_k^j$.
Then $a_l^i + b_l^i - b_1^i = x = a_k^j + b_k^j - b_1^j$ so  
\begin{equation}\label{star 1}
a_l^i + b_l^i + b_1^j = a_k^j + b_k^j + b_1^i ~ \textrm{for any} ~ 2 \leq l \leq p_i ~ \textrm{and} ~ 2 \leq k \leq p_j.
\end{equation}
If $b_1^j \in T_i^1$ then there is no loss in assuming $b_1^j \in \{ a_2^i , b_2^i \}$.  The same assumption may be made with $i$ and $j$ interchanged.  This means that for $l \geq 3$, $b_1^j$ is not a term in the sum 
$a_l^i + b_l^i$ and for $k \geq 3$, $b_1^i$ is not a term in the sum $a_k^j + b_k^j$.  The $B_3^*$ property and (\ref{star 1}) imply 
\begin{equation}\label{star 2}
| \{ a_l^i , b_l^i \} \cap \{ a_k^j , b_k^j \} | =1 ~ \textrm{for any} ~ 3 \leq l \leq p_i ~ \textrm{and} ~ 3 \leq k \leq p_j.
\end{equation}
In particular, $\{a_3^i , b_3^i \} \cap \{ a_3^j , b_3^j \} \neq \emptyset$ 
and $\{a_3^i , b_3^i \} \cap \{a_4^j , b_4^j \} \neq \emptyset$ so that $p_j \leq 4$.  Here we are using the fact that 
any element of $A$ can occur at most once in the list $a_1^i, b_1^i,  \dots , a_{p_i}^i, b_{p_i}^i$.  By symmetry, $p_i \leq 4$.  

If $p_i = p_j =4$ then by (\ref{star 2}), $\{ a_3^i , b_3^i , a_4^i , b_4^i \} = \{ a_3^j , b_3^j , a_4^j , b_4^j \}$
but then $2d_i = a_3^i  + b_3^i + a_4^i + b_4^i = 2d_j$ implying $d_i = d_j$, a contradiction.     
\end{proof}

\begin{corollary}\label{star cor}
If $p_i \geq 4$ and $p_j \geq 4$ with $i \neq j$ then $Q_i^1 \cap Q_j^1 = \emptyset$.
\end{corollary}

\vspace{1em}

With our notation, we can write 
\begin{equation*}
\sum r_2 (n)^2 = \sum_{i=1}^{M} |P_i^2|^2 = m_1 + 4(m_2 - m_1) + 9(m_3 - m_2) + 16 (m_4 - m_3) + \sum_{i=m_4 +1}^{M} 
|P_i^2|^2.
\end{equation*}
If $p_i = p_j = 4$ for some $i \neq j$ then $Q_i^1 \cap Q_j^1 = \emptyset$ by Corollary~\ref{star cor} so 
$m_4 - m_3 \leq \frac{ |A|}{8}$.  For $1 \leq i \leq 3$, let $\delta_i |A|^2 = m_i - m_{i-1}$.  Then 
\begin{equation}\label{1.1 eq 2}
\sum r_2 (n)^2 \leq |A|^2 ( \delta_1 + 4 \delta_2 + 9 \delta_3) + \sum_{i=m_4+1}^{M} |P_i^2|^2 + 2|A|.
\end{equation}

Define a graph $H$ with vertex set $Q_{m_2+1}^1 \cup \dots \cup Q_{m_3}^1$ and edge set 
$P_{m_2 +1}^2 \cup \dots \cup P_{m_3}^2$.  Let $n = |V(H)|$.  $H$ has $3(m_3 - m_2) = 3 \delta_3 |A|^2$ edges so 
$3 \delta_3 |A|^2 \leq \frac{n}{2}$ which can be rewritten as 
\begin{equation}\label{1.1 eq 3}
\sqrt{6 \delta_3} |A| \leq | Q_{m_2+1}^1 \cup \dots \cup Q_{m_3}^1 |.
\end{equation}
For any $i$ and $j$ with $m_2 +1 \leq i \leq m_3$ and $m_4 +1 \leq j \leq M$, $Q_i^1 \cap Q_j^1 = \emptyset$ by Lemma~\ref{1.1 lemma 8} so that (\ref{1.1 eq 3}) implies 
\begin{equation*}
\sum_{i=m_4 +1}^{M} |P_i^2| = \frac{1}{2} \sum_{i=m_4 +1}^{M} |Q_i^1| = \frac{1}{2} 
|Q_{m_4 +1}^1 \cup \dots \cup Q_M^1 | \leq \frac{1}{2} (1 - \sqrt{6 \delta_3 } )|A|.
\end{equation*}
We conclude $\sum_{i=m_4 +1}^{M} | P_i^2 |^2 \leq \left( \frac{1 - \sqrt{6 \delta_3} }{2} \right)^2 |A|^2$.  This estimate and (\ref{1.1 eq 2}) give 
\begin{equation}\label{1.1 eq 4}
\sum r_2 (n)^2 \leq   |A|^2 \left( \delta_1 + 4 \delta_2 + 9 \delta_3 + \frac{1}{4}( 1 - \sqrt{6 \delta_3} )^2 \right)  + 2|A|.
\end{equation}

Each pair $\{a,b\} \in A^{ (2)}$ is in at most one $P_i^2$ so 
\begin{equation*}
|A|^2 ( \delta_1 + 2 \delta_2 + 3 \delta_3 ) = m_1 + 2(m_2 - m_1) + 3 (m_3 - m_2) \leq { |A| \choose 2} \leq \frac{ |A|^2 }{2}.
\end{equation*}
The maximum of $\delta_1 + 4 \delta_2 + 9 \delta_3 + \frac{1}{4} ( 1 - \sqrt{ 6 \delta_3 } )^2$ subject to the conditions 
$\delta_1 + 2 \delta_2 + 3 \delta_3 \leq \frac{1}{2}$, $\delta_1 \geq 0$, $\delta_2 \geq 0$, and $\delta_3 \geq 0$ is $\frac{3}{2}$ obtained when 
$\delta_1 = \delta_2 = 0$ and $\delta_3 = \frac{1}{6}$.  By (\ref{1.1 eq 4}), 
\begin{equation}\label{1.1 eq 4.5}
\sum r_2 (n)^2 \leq \frac{3 |A|^2 }{2} + 2 |A|.
\end{equation}
An immediate consequence is that 
\begin{equation}\label{1.1 eq 5}
\sum_{n \in 2 \cdot A} r_2 (n) = \sum \ind_{2 \cdot A}(n) r_2 (n) \leq |A|^{1/2} \left( \sum r_2 (n)^2 \right)^{1/2} \leq 2 |A|^{3/2}.
\end{equation}
Next we proceed as in Lemma~\ref{1.1 lemma 7}.  Using (\ref{1.1 eq 5}) and (\ref{1.1 eq 4.5}), 
\begin{eqnarray*}
\sum \sigma_2 (n)^2 & = & 4 \sum r_2 (n)^2 + 4 \sum_{n \in 2 \cdot A} r_2 (n) + | 2 \cdot A| \\
& \leq & 6 |A|^2 + 8 |A|^{3/2} + 9 |A|.
\end{eqnarray*}
To finish the proof, recall
\begin{equation*}
\sum \sigma_3 (n)^2 \leq (1 + o(1)) 3^2 |A| \sum \sigma_2 (n)^2
\end{equation*}
and now we use our upper bound on $\sum \sigma_2 (n)^2$ to get $c_3^* \leq 54$.  
\end{proof}  

\begin{corollary}
If $A \subset [N]$ is a $B_3^*$-set then 
\begin{equation*}
|A| \leq (1 + o(1))( 162 N)^{1/3}.
\end{equation*}
\end{corollary}
\begin{proof}
If $A \subset [N]$ is a $B_3^*$-set then 
\begin{equation*}
\frac{ |A|^6 }{3N } \leq \sum \sigma_3 (n)^2 \leq (1 + o(1)) 54 |A|^3.
\end{equation*}
\end{proof}

\vspace{1em}

Thus far we have shown $c_2^+ \leq c_2^* \leq 2$, $c_3^+ \leq 18$, and $c_3^* \leq 54$.  Now we describe our method for obtaining upper bounds on $F_k^+ (N)$ and $F_k^* (N)$.  Assume we have upper bounds on $c_2^+, c_3^+ , \dots , c_{k-1}^+$.  Lemma~\ref{1.1 lemma 5} gives an upper bound on $|A|$ in terms of $c_{k/2}^+$ when $k$ is even, and in terms of $c_l^+$ and $c_{l+1}^+$ when $k=2l+1 \geq 5$.  An upper bound on $c_k^+$ is obtained from (\ref{1.2 even}) and (\ref{1.2 odd}).  We can also apply this method to $B_k^*$-sets.  The bounds we obtain are given in Table 1 below.  They have been rounded up to the nearest tenth and they hold for large enough $N$ without error terms.         

\begin{center}

\begin{tabular}{|c||c||c|c|} \hline
$k$ 		& Upper Bound of \cite {R}	& Our Upper Bound on $F_k^*$   & Our Upper Bound on $F_k^+$   \\ \hline
3  		& $6.3N^{1/3}$ 			& $5.5N^{1/3}$		& $2.7N^{1/3}$ \\
4     		& $11.4  N^{1/4}$			& $6.8 N^{1/4}$   & $4.1 N^{1/4}$		\\  
5		& $18.2 N^{1/5}$  			& $11.2 N^{1/5}$	& $11N^{1/5}$	\\  
6   		& $26.8 N^{1/6}$            		& $15.8 N^{1/6}$	& $13.1N^{1/6}$	  \\      
7  		& $37.2  N^{1/7}$              	& $21.6 N^{1/7}$	& $18.5N^{1/7}$		  \\ 
8  		& $49.4  N^{1/8}$              	& $22.7 N^{1/8}$	& $22.7N^{1/8}$	  \\ \hline 
\end{tabular}

\vspace{2mm}

Table 1: Upper bounds on $B_k^+$-sets and $B_k^*$-sets.

\end{center} 

We conclude this section with our proof of the second statement of Theorem~\ref{upper bounds}.  Recall (\ref{1.2 even}) states   
$c_k^* \leq k^{k} c_{k/2}^*$ for any even $k \geq 4$, and (\ref{1.2 odd}) gives $c_k^* \leq k^{k+1} \max \{c_l^* , 
c_{l+1}^* \}$ for $ k= 2l+1 \geq 5$.  
For $x \geq 0$ let $\lceil x \rceil$ be the smallest integer greater than or equal to $x$ and let $\lfloor x \rfloor$ be the greatest integer less than or equal to $x$.  For $k \geq 0$, define $\phi_1 (k) = \lceil \frac{k}{2} \rceil$ and 
$\phi_i (k) := \phi_{1} ( \phi_{i-1} (k ) )$ for $i \geq 2$.  A simple induction argument can be used to show that for all $i \geq 1$, 
$\phi_i (k) \leq k2^{-i} + \sum_{t=0}^{i-1} 2^{-t}$.  The conclusion is that for every $i \geq 1$, 
$\phi_i (k) \leq k2^{-i} + 2$.  For any $k \geq 5$, 
\begin{equation*}
c_k^*  \leq  k^{k+1} \prod_{i=1}^{ \lfloor \log_2 k \rfloor} \phi_i (k)^{ \phi_i (k) +1 } 
 \leq   k^{k+1} \prod_{i=1}^{ \lfloor \log_2 k \rfloor } \left( k2^{-i} +2 \right)^{ k2^{-i} +3 }.
\end{equation*}   
Taking $k$-th roots, 
\begin{eqnarray*}
(c_k^*)^{1/k} & \leq & k^{1+1/k} \prod_{ i =1}^{ \lfloor \log_2 k \rfloor} (k2^{-i} + 2)^{2^{-i} + 3/k} 
\leq k^{1+1/k} \left( \frac{k}{2} + 2 \right)^{ \frac{ 3 \log_2 k }{k} } \prod_{i=1}^{ \lfloor \log_2 k \rfloor} (k2^{-i} +2)^{2^{-i} } \\
& \leq & k^{1 + 1/k} k^{ \frac{ 3 \log_2 k }{k} } k^{ \sum_{i=1}^{ \lfloor \log_2 k \rfloor } 2^{-i} } 
\prod_{i=1}^{ \lfloor \log_2 k \rfloor } \left( 2^{-i} + \frac{2}{k} \right)^{2^{-i} } \\
& \leq & k^2 k^{ \frac{4 \log_2 k }{k} } \prod_{i=1}^{ \lfloor \log_2 k \rfloor } \left(2^{-i} + \frac{2}{k} \right)^{2^{-i}}.
\end{eqnarray*}

We claim the sequence $(c_{k}^*)^{1/k}$ is bounded above by a function $F(k)$ that tends to $\frac{k^2}{4}$ as $k \rightarrow \infty$.  With this in mind, we rewrite the previous inequality as 
\begin{equation}\label{1.1 eq 10}
\frac{ 4 (c_k^*)^{1/k} }{k^2} \leq 4 k^{ \frac{4 \log_2 k }{k} } \prod_{i=1}^{ \lfloor \log_2 k \rfloor} \left( 2^{-i} + \frac{2}{k} 
\right)^{2^{-i} }.
\end{equation} 
It is easy to check $k^{ \frac{ 4 \log_2 k }{k} } \rightarrow 1$ as $k \rightarrow \infty$.  
Using $\sum_{n=0}^{\infty} nx^{n-1} = \frac{1}{(1-x)^2}$ from elementary calculus,  
\begin{equation*}
\prod_{i=1}^{ \lfloor \log_2 k \rfloor} (2^{-i})^{2^{-i}} = \left( \frac{1}{2} \right)^{ \sum_{i=1}^{\lfloor \log_2 k \rfloor} i2^{-i} } 
\rightarrow \frac{1}{4}
\end{equation*}
as $k \rightarrow \infty$.   
Using the inequality $1+x \leq e^x$ for $x \geq 0$ we have 
\begin{eqnarray*}
1 & \leq &  \frac{    \prod_{i=1}^{ \lfloor \log_2 k \rfloor} (2^{-i} + 2/k)^{2^{-i} } }{  \prod_{i=1}^{ \lfloor \log_2 k \rfloor} 
(2^{-i} )^{2^{-i} }} = \prod_{i=1}^{ \lfloor \log_2 k \rfloor} \left( 1 + \frac{2^{i+1}}{k} \right)^{2^{-i} } \\
& \leq & \prod_{i=1}^{ \lfloor \log_2 k \rfloor} e^{ 2^{i+1} /k} \leq e^{ \frac{1}{k} \sum_{i=2}^{ \lfloor \log_2 k \rfloor}2^i } \leq e^{1/k}.
\end{eqnarray*} 
As $k \rightarrow \infty$, $e^{1/k} \rightarrow 1$ so 
\begin{equation*}
\prod_{i=1}^{ \lfloor \log_2 k \rfloor} \left( 2^{-i} + \frac{2}{k} \right)^{2^{-i} } \rightarrow \frac{1}{4}.
\end{equation*}
This shows that the right hand side of (\ref{1.1 eq 10}) tends to 1 and $k \rightarrow \infty$ which proves the claim.  

Given $\epsilon >0$, we can choose $k$ large enough so that $k^{1/k} (c_k^*)^{1/k} \leq (1 + \epsilon ) \frac{k^2}{4}$.  The theorem now follows from the definition of $c_k^*$ and the estimate $\frac{ |A|^{2k} }{kN} \leq \sum \sigma_k (n)^2$.

%%%~%%%%%%%%%%%%%%%%%%%%%%%%%%%%%

\section{Proof of Theorem~\ref{lb 4}}

%%%%%%%%%%%%%%%%%%%%%%%%%%%%%%%%%%

\begin{lemma}\label{lemma 5.1}
If $A \subset G$ is a non-abelian $B_k$-set and $B \subset H$ is a non-abelian $B_k^+$-set then $A \times B$ is a 
non-abelian $B_k^+$-set in $G \times H$.
\end{lemma}
\begin{proof}
Suppose $a_1, \dots , a_k, a_{1}', \dots , a_{k}' \in A$, $b_1, \dots , b_k, b_{1}', \dots , b_{k}' \in B$ and 
\begin{equation*}
(a_1, b_1 ) \cdots (a_k,b_k) = (a_{1}', b_{1}') \cdots (a_{k}' , b_{k}').
\end{equation*}
Then $a_1 \cdots a_k = a_{1}' \cdots a_{k}'$ and $b_1 \cdots b_k = b_{1}' \cdots b_{k}'$ so $a_i = a_{i}'$ for every $i$ and 
$b_j = b_{j}'$ for some $j$ hence $(a_j , b_j ) = (a_{j}' , b_{j}' )$.  
\end{proof}

\vspace{1em}

Let $\mathbb{F}_4 = \{0 , 1, a, b \}$ be the finite field with four elements and
\begin{equation*}
H = \left\{ 
\left( \begin{array}{cc}
x & y \\
0 & x^{-1} \end{array} \right) :
x \in \mathbb{F}_{4}^{*}, y \in \mathbb{F}_4 \right\}.   
\end{equation*}
$H$ is a group under matrix multiplication and $|H| = 12$.  Let 
\begin{equation*}
\alpha = \left( \begin{array}{cc}
a & 1 \\
0 & b \end{array} \right) ~~ \textrm{and} ~~
\beta = \left( \begin{array}{cc}
a & a \\
0 & b \end{array} \right).
\end{equation*}

Simple computations show that $\alpha$ and $\beta$ satisfy $\alpha^3 = \beta^3 = \textrm{id}$ and 
$\alpha^2 \beta = \beta^2 \alpha$.  

\begin{lemma}\label{AB lemma}
The set $\{\alpha, \beta \}$ is a $B_4^+$-set in $H$.    
\end{lemma}
\begin{proof}
Suppose there is a solution to the equation $x_1 x_2 x_3 x_4 = y_1 y_2 y_3 y_4$ with 
$x_i \neq y_i$ for $1 \leq i \leq 4$ and $x_i,y_j \in \{ \alpha , \beta \}$ for all $i,j$.  Without loss of generality, assume $x_1 = \alpha$ and $y_1 = \beta$.  There are eight cases which we can deal with using the relations $\alpha^3 = \beta^3=\textrm{id}$ and $\alpha^2 \beta = \beta^2 \alpha$.  Instead of considering each individually, we handle several cases at the same time.  

\noindent
\textbf{Case 1:}  $\alpha^4 = \beta^4$ or $\alpha^3 \beta = \beta^3 \alpha$ or $\alpha \beta^3 = \beta \alpha^3$.

If any of these equations hold then the relation $\alpha^3 = \beta^3 = \textrm{id}$ implies $\alpha = \beta$, a contradiction.

\noindent
\textbf{Case 2:}  $\alpha^2 \beta \alpha = \beta^2 \alpha \beta$ or $\alpha^2 \beta^2 = \beta^2 \alpha^2$.     

If either of these equations hold then the relation $\alpha^2 \beta = \beta^2 \alpha$ implies $\alpha = \beta$.   

\noindent
\textbf{Case 3:}  $\alpha \beta \alpha^2 = \beta \alpha \beta^2$.  

Multiplying the equation on the right by $\beta$ and using $\beta^3 = \textrm{id}$, we get $\alpha \beta \alpha^2 \beta = \beta \alpha$.  On the other hand,
$\alpha \beta \alpha^2 \beta = \alpha \beta^3 \alpha = \alpha^2$ so combining the two equations we get $\beta \alpha = \alpha^2$ which implies $\alpha = \beta$, a contradiction.

\noindent
\textbf{Case 4:} $\alpha \beta \alpha \beta = \beta \alpha \beta \alpha$.

Multiply the equation on the left by $\beta^2$ to get $\beta^2 \alpha \beta \alpha \beta = \alpha \beta \alpha$ which can be 
rewritten as $\alpha^2 \beta^2 \alpha \beta = \alpha \beta \alpha$ using $\beta^2 \alpha = \alpha^2 \beta$.  
Replace $\beta^2 \alpha$ with $\alpha^2 \beta$ on the left hand side of $\alpha^2 \beta^2 \alpha \beta = \alpha \beta \alpha$ and 
cancel $\alpha$ to get $\beta^2 = \beta \alpha $ which implies $\beta = \alpha$.  

%A simple computation shows 
%\begin{equation*}
%\alpha \beta \alpha \beta = \left( \begin{array}{cc}
%a & 0 \\
%0 & b \end{array} \right) \neq
%\left( \begin{array}{cc}
%a & b \\
%0 & b \end{array} \right) = \beta \alpha \beta \alpha.
%\end{equation*}

\noindent
\textbf{Case 5:}  $\alpha \beta^2 \alpha = \beta \alpha^2 \beta$.  

Using the relation $\beta^2 \alpha = \alpha^2 \beta$, we can rewrite this equation as 
$\alpha^3 \beta = \beta^3 \alpha$ which implies $\alpha = \beta$ since $\alpha^3 = \beta^3 = \textrm{id}$.

\end{proof}

\vspace{1em}

The set $\{\alpha , \beta \}$ is not a non-abelian $B_4$-set since $\alpha^2 \beta \beta = \beta^2 \alpha \beta$.   
The next theorem is a special case of a result of Odlyzko and Smith and we will use it in our construction.

%Odlyzko and Smith \cite{OS} define a \emph{non-abelian $B_k$-set} $A \subset G$, $G$ a finite non-abelian group, to be a set such that all $k$-letter %words 
%whose letters are chosen with replacement from $A$ are distinct.  They show that for each integer $k \geq 2$ and prime $p$ with $k$ dividing $p-1$, %there is a 
%non-abelian group $G$ of order $k(p^k - 1)$ and a non-abelian $B_k$-set $A \subset G$ such that $|A| = \frac{1}{k}(p-1)$.  

\begin{theorem}[Odlyzko, Smith, \cite{OS}]\label{OS theorem}
For each prime $p$ with $p-1$ divisible by 4, there is a non-abelian group $G$ of order $4(p^4 - 1)$ and a non-abelian $B_4$-set $A \subset G$ with \begin{equation*}
|A| = \frac{1}{4}(p-1).
\end{equation*}   
\end{theorem}

%Taking Odlyzko and Smith's definition as motivation, we define a \emph{non-abelian $B_k^*$-set} to be a set $A \subset G$ such that whenever
%\begin{equation*}
%a_1 a_2 \cdots a_k = b_1 b_2 \cdots b_k ~ \textrm{with} ~ a_1, \dots , a_k , b_1 , \dots , b_k \in A,
%\end{equation*}
%we have $a_i = b_i$ for some $1 \leq i \leq k$.  As before, a non-abelian $B_k$-set is also a non-abelian $B_k^*$-set but the converse does not hold %in general.
%For example, Lemma~\ref{AB lemma} implies $\{\alpha,\beta \}$ is a non-abelian $B_4^*$-set but it is not a non-abelian $B_4$-set since $\alpha^2 \beta %\beta = \beta^2 \alpha \beta$.  Using the construction of \cite{OS}, we can show that there exists infinitely many finite 
%non-abelian groups $G$ such that $G$ has a non-abelian $B_4^*$-set $A \subset G$ with 
%\begin{equation*}
%|A| = \left( \frac{|G|}{768} \right)^{1/4}  + o( |G|^{1/4}).
%\end{equation*}  
%The non-abelian $B_4$-sets constructed by Odlyzko and Smith have $\left( \frac{|G|}{1024} \right)^{1/4} + o( |G|^{1/4})$ elements.  

%\begin{theorem}
%For any prime $p$ with $p-1$ divisible by 4, there exists a non-abelian group $G$ of order $12 \cdot 4(p^4 -1)$ that contains a non-abelian %$B_4^*$-set $A \subset G$ with 
%\begin{equation*}
%|A| = \frac{1}{2} ( p -1).
%\end{equation*}
%\end{theorem}
%\begin{proof}

Armed with Lemma~\ref{lemma 5.1}, Lemma~\ref{AB lemma}, and Theorem~\ref{OS theorem} it is easy to prove Theorem~\ref{lb 4}.  

Let $p$ be any prime with $p-1 $ divisible by 4.  By Theorem~\ref{OS theorem}, there is a group $G_1$ of order $4(p^4 - 1)$ and 
a non-abelian $B_4$-set $A_1 \subset G_1$ with $|A_1 | = \frac{1}{4}(p-1)$.  
Define the group $G$ to be the product group $G = G_1 \times H$.  Let $A = A_1 \times \{\alpha,\beta \}$.  Clearly 
$|G| = 12 \cdot 4 (p^4 -1)$, $|A| = \frac{1}{2} (p-1)$, and by Lemma~\ref{lemma 5.1}, $A$ is a non-abelian $B_4^+$-set in $G$.  

%It remains to show that $A$ is a non-abelian $B_4^*$-set.  If
%\begin{equation*}
%(a_1, x_1)(a_2, x_2)(a_3,x_3)(a_4,x_4) = (b_1, y_1)(b_2, y_2)(b_3,y_3)(b_4,y_4)
%\end{equation*}
%with $(a_i , x_i), (b_j, y_j) \in A$, then $a_1 a_2 a_3 a_4 = b_1 b_2 b_3 b_4$ and $x_1 x_2 x_3 x_4 = y_1 y_2 y_3 y_4$.  Since $A_1$ is a non-abelian %$B_4$-set, 
%$a_i = b_i$ for $1 \leq i \leq 4$.  By Lemma~\ref{AB lemma}, $x_j = y_j$ for some $j \in \{1,2,3,4 \}$ thus  
%$(a_j , x_j) = (b_j , y_j)$ and $A$ is a non-abelian $B_4^*$-set.
%\end{proof}   

%%%%%%%%%%%%%%%%%%%%%%%%%%%%%%%%

%\section{Concluding Remarks}

%$\bullet$ We whether or not $F_k^+ (N)$ is asymptotic to $F_k^* (N)$ for any $k \geq 3$.  It would be interesting to determine %whether or not this was true.  Fix a $k \geq 3$.  A $B_k^*$-set can have solutions to the equation  
%$2x_1 + x_2 + \cdots x_{k-1} = y_1 + y_2 + \cdots y_{k}$ in $2k-1$ distinct integers but a $B_k^+$-set cannot.  
%Using the ideas of \cite{R}, it can be shown that if $A$ is a $B_k^*$-set then the number of solutions to this equation 
%is $o(|A|^k)$ so one might expect that $F_k^+(N)$ would be asymptotic to $F_k^* (N)$ but we were unable to show this.  This %issue is also discussed briefly in \cite{OB}.    

%%%%%%%%%%%%%%%%%%%%%%%%%%%%%

\section{Acknowledgement} 

The author would like to thank Jacques Verstra\"{e}te for helpful comments.  

%%%%%%%%%%%%%%%%%%%%%%%%%%%%%%%%%%

\end{document}